\documentclass{article}
\usepackage[ansinew]{inputenc}          
\usepackage{amssymb,amsthm,amsmath}     
\usepackage{psfrag}                     
\usepackage[dvips]{graphicx,gincltex}            
\usepackage[all]{xy}
\usepackage{subfig}
\usepackage{algorithm,algorithmic}
\usepackage{a4wide}
\usepackage{caption}
\usepackage{centernot}
\usepackage{array,colortbl,xcolor}

\theoremstyle{plain}
\newtheorem{theorem}{Theorem}

\newtheorem{lemma}[theorem]{Lemma}
\newtheorem{corollary}[theorem]{Corollary}
\theoremstyle{definition}
\newtheorem{definition}[theorem]{Definition}
\newtheorem{example}[theorem]{Example}

\newtheorem{observation}[theorem]{Observation}
\newtheorem{conjecture}[theorem]{Conjecture}

\newcommand{\Z}{\mathbb{Z}}
\newcommand{\F}{\mathbb{F}}

\newcommand{\Ext}{\text{Ext}}

\newcommand{\nolla}{\mathbf{0}}

\usepackage[dvips]{epsfig}
\usepackage{tikz}
\usetikzlibrary{matrix}
\usetikzlibrary{positioning}
\tikzset{main node/.style={circle,fill=blue!20,draw,minimum size=0.2cm,inner sep=0pt},
            }

\begin{document}

\tikzset{square matrix/.style={
    matrix of nodes,
    column sep=-\pgflinewidth, row sep=-\pgflinewidth,
    nodes={draw,
      minimum height=#1,
      anchor=center,
      text width=#1,
      align=center,
      inner sep=0pt
    },
  },
  square matrix/.default=0.5cm
}

\tikzset{bsquare matrix/.style={
    matrix of nodes,
    column sep=-\pgflinewidth, row sep=-\pgflinewidth,
    nodes={draw,
      minimum height=#1,
      anchor=center,
      text width=#1,
      align=center,
      inner sep=0pt
    },
  },
  bsquare matrix/.default=0.6cm
}

\tikzset{big square matrix/.style={
    matrix of nodes,
    column sep=-\pgflinewidth, row sep=-\pgflinewidth,
    nodes={draw,
      minimum height=#1,
      anchor=center,
      text width=#1,
      align=center,
      inner sep=0pt
    },
  },
  big square matrix/.default=0.88cm
}

\tikzset{huge square matrix/.style={
    matrix of nodes,
    column sep=-0, row sep=-0,
    nodes={draw,
      minimum height=#1,
      anchor=center,
      text width=#1,
      align=center,
      inner sep=0pt
    },
  },
  huge square matrix/.default=3.52cm
}

\title{Solving a Conjecture on Identification in Hamming Graphs\footnote{A shortened version \cite{JLLICGT18} of the paper has been accepted to the 10th International Colloquium on Graph Theory and combinatorics, Lyon, 2018.}}

\author{\textbf{Ville Junnila}, \textbf{Tero Laihonen} and \textbf{Tuomo Lehtil{\"a}}\thanks{Research supported by the University of Turku Graduate School (UTUGS).}\\
Department of Mathematics and Statistics\\
University of Turku, FI-20014 Turku, Finland\\
viljun@utu.fi, terolai@utu.fi and tualeh@utu.fi}

\date{}
\maketitle

\begin{abstract}
Identifying codes in graphs have been widely studied since their introduction by Karpovsky, Chakrabarty and Levitin in 1998. In particular, there are a lot of results regarding the binary hypercubes, that is, the Hamming graphs $K_2^n$. In 2008, Gravier  \emph{et al.} started investigating identification in $K_q^2$. Goddard and Wash, in 2013, studied identifying codes in the  general Hamming graphs $K_q^n$. They stated, for instance, that $\gamma^{ID}(K_q^n)\leq q^{n-1}$ for any $q$ and $n\geq3$. Moreover, they conjectured that $\gamma^{ID}(K_q^3)=q^2$.  In this article, we show that $\gamma^{ID}(K_q^3)\leq q^2-q/4$ when $q$ is a power of four, disproving the conjecture.  Our approach is based on the recursive use of suitable designs. Goddard and Wash also gave  the following lower bound  $\gamma^{ID}(K_q^3)\ge  q^2-q\sqrt{q}$. We  improve this bound to $\gamma^{ID}(K_q^3)\ge   q^2-\frac{3}{2} q$. The conventional methods used for obtaining lower bounds on identifying codes do not help here. Hence, we provide a different technique building on the approach of Goddard and Wash. Moreover, we improve the above mentioned bound $\gamma^{ID}(K_q^n)\leq q^{n-1}$ to $\gamma^{ID}(K_q^n)\leq q^{n-k}$ for $n=3\frac{q^k-1}{q-1}$ when $q$ is a prime power. For this bound, we utilize suitable linear codes over finite fields and a class of closely related codes, namely, the self-locating-dominating codes. In addition, we show that the self-locating-dominating codes satisfy the result
 $\gamma^{SLD}(K_q^3)=q^2$ related to the above conjecture.
\end{abstract}


\section{Introduction}

Sensor networks are systems consisting of sensors and links between them. As a monitoring tool they may be used for example in surveillance or to oversee arrays of processors. The basic idea is that a sensor is placed at some node of a network and then it monitors its surroundings reporting on possible anomalies or irregularities. Based on these reports the central unit will deduce the location of the anomaly. The goal is to minimize the number of sensors in networks with certain structures. More  on location in sensor networks can be found in~\cite{Trachtenberg,LT:disj,Ray}.

 A simple, undirected and connected graph $G=(V,E)$ is utilized to model the sensor network. The set of vertices adjacent to a vertex $v$ is called the \textit{open neighbourhood} of $v$, denoted by $N(v)$, and the set $N(v)\cup \{v\}=N[v]$ is called the \textit{closed neighbourhood} of $v$. The vertices represent the possible locations of the sensors and the edge set determines the surrounding area of a sensor. In other words, a sensor placed on vertex $u$ monitors  locations $N[u]$ for irregularities.

 A nonempty subset $C$ of a vertex set $V$ of a graph is called a \textit{code} and its elements are called \textit{codewords}. We define the \emph{ identifying set} or $I$-\emph{set} of a vertex $u$ as $$I(u)=N[u]\cap C$$ or if the code or graph is unclear we may use the notation $I(C;u)$ or $I(G,C;u)$. The $I$-set can also be defined for a \emph{set} of vertices. That is, for $U\subseteq V$, we define $$I(U)=\bigcup_{v\in U}I(v).$$ In this paper, the code can be understood as the locations of the sensors within our sensor network and the $I$-set of $u$ as the set of sensors which oversee the location $u$. The set of vertices $C$ is called a \textit{dominating set} if $I(v)\neq \emptyset$ for each vertex $v\in V$ and the minimum size of a dominating set in a graph $G$ is called the \textit{domination number} $\gamma(G)$. Hence, if sensors are placed at vertices which form a dominating set, then each location is monitored by a sensor and an irregularity is always detected. However, if the sensors report only that there is an irregularity within the area they monitor, then we need stronger condition than just a dominating set for locating the irregularity. For this purpose Karpovsky, Chakrabarty and Levitin defined \textit{identifying codes} in \cite{kcl}. More  on identifying codes can be found at \cite{lowww} and for recent development, see \cite{IDHereditary},\cite{IDvTrans} and\cite{IDMoreCompl}.
\begin{definition}
A code $C\subseteq V$ is \textit{identifying} in a graph $G$ if $C$ is a dominating set and $$I(u)\neq I(v)$$ for each pair of distinct vertices $u,v\in V$. An identifying code $C$ of minimum cardinality in a finite graph $G$ is called \textit{optimal} and its cardinality is denote with $\gamma^{ID}(G)$.
\end{definition}

The previous definition of identifying code is illustrated in the following example.
\begin{example}
Let us consider graph $G$ of Figure \ref{FigureID} and the code $C=\{a,b,c\}$. We have $I(d)=\{a\}$, $I(e)=\{b\}$, $I(f)=\{c\}$, $I(a)=\{a,b\}$, $I(b)=\{a,b,c\}$ and $I(c)=\{b,c\}$. Hence, each $I$-set is non-empty and unique and, therefore, the code $C$ is an identifying code. Moreover, there are no smaller identifying codes in $G$ since using at most two codewords we can form at most three different nonempty subsets of the code. Hence, $C$ is an optimal identifying code in $G$ and $\gamma^{ID}(G)=3$.
\end{example}

Notice that there are some possible problems with identifying codes if more than one irregularity may occur in the sensor network. For instance, in the previous example, we have $I(b) = I(\{d,e,f\})$. Hence, if there were irregularities in the vertices $d$, $e$ and $f$, then we would mistakenly deduce that an irregularity is in the vertex $b$. Moreover, we would not even notice that something went wrong. To overcome this problem, in~\cite{+koodithl}, so called self-identifying codes, which are able to locate one irregularity and detect multiple ones, have been introduced. (Notice that in the original paper self-identifying codes are called $1^+$-identifying.) The formal definition of self-identifying codes is given as follows.


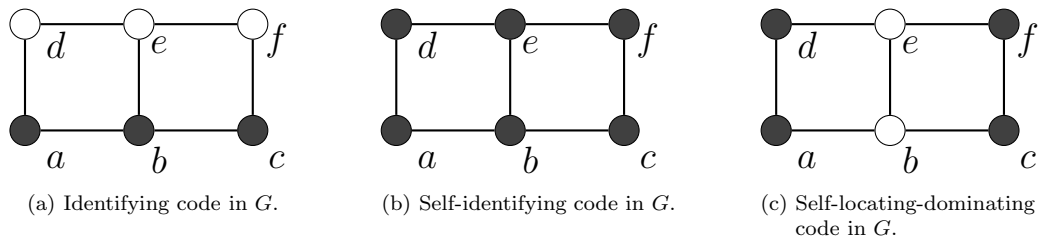
\begin{figure}[ht]
\captionsetup[subfigure]{format=hang,justification=raggedright,singlelinecheck=on}
    \centering
 \subfloat[Identifying code in $G$.]{\begin{tikzpicture}
    \node[main node, minimum size=0.4cm] (1)[fill=darkgray]  {};
    \node[main node, minimum size=0.4cm] (2) [above = 1cm of 1][fill=white]  {};
    \node[main node, minimum size=0.4cm] (3) [right = 1.1cm of 1][fill=darkgray]    {};
    \node[main node, minimum size=0.4cm] (4) [right = 1.1cm of 3][fill=darkgray]  {};
    \node[main node, minimum size=0.4cm] (5) [above = 1cm of 3][fill=white]   {};
    \node[main node, minimum size=0.4cm] (6) [above = 1cm of 4][fill=white]   {};

    \node (7) [below right = 0.01cm and 0.01cm of 1]  {\Large $a$};
    \node (8) [above = 1cm of 7]  {\Large $d$};
    \node (9) [right = 0.89cm of 7]    {\Large $b$};
    \node (10) [right = 1.10cm of 9]  {\Large $c$};
    \node (11) [above = 1cm of 9]   {\Large $e$};
    \node (12) [above = 1cm of 10]   {\Large $f$};

    \path[draw,thick]
    (1) edge node {} (2)
    (1) edge node {} (3)
    (3) edge node {} (4)
    (3) edge node {} (5)
    (5) edge node {} (6)
    (4) edge node {} (6)
    (2) edge node {} (5)

;
\end{tikzpicture}\label{FigureID}}
 \hspace{1cm}
\subfloat[Self-identifying code in $G$.]{\begin{tikzpicture}
    \node[main node, minimum size=0.4cm] (1)[fill=darkgray]  {};
    \node[main node, minimum size=0.4cm] (2) [above = 1cm of 1][fill=darkgray]  {};
    \node[main node, minimum size=0.4cm] (3) [right = 1.1cm of 1][fill=darkgray]   {};
    \node[main node, minimum size=0.4cm] (4) [right = 1.1cm of 3][fill=darkgray]  {};
    \node[main node, minimum size=0.4cm] (5) [above = 1cm of 3][fill=darkgray]   {};
    \node[main node, minimum size=0.4cm] (6) [above = 1cm of 4][fill=darkgray]   {};

    \node (7) [below right = 0.01cm and 0.01cm of 1]  {\Large $a$};
    \node (8) [above = 1cm of 7]  {\Large $d$};
    \node (9) [right = 0.89cm of 7]    {\Large $b$};
    \node (10) [right = 1.10cm of 9]  {\Large $c$};
    \node (11) [above = 1cm of 9]   {\Large $e$};
    \node (12) [above = 1cm of 10]   {\Large $f$};

    \path[draw,thick]
    (1) edge node {} (2)
    (1) edge node {} (3)
    (3) edge node {} (4)
    (3) edge node {} (5)
    (5) edge node {} (6)
    (4) edge node {} (6)
    (2) edge node {} (5)

;
\end{tikzpicture}
\label{FigureSID}}
 \hspace{1cm}
\subfloat[Self-locating-dominating code in $G$.]{\begin{tikzpicture}
    \node[main node, minimum size=0.4cm] (1)[fill=darkgray]  {};
    \node[main node, minimum size=0.4cm] (2) [above = 1cm of 1][fill=darkgray]  {};
    \node[main node, minimum size=0.4cm] (3) [right = 1.1cm of 1][fill=white]    {};
    \node[main node, minimum size=0.4cm] (4) [right = 1.1cm of 3][fill=darkgray]  {};
    \node[main node, minimum size=0.4cm] (5) [above = 1cm of 3][fill=white]   {};
    \node[main node, minimum size=0.4cm] (6) [above = 1cm of 4][fill=darkgray]   {};

    \node (7) [below right = 0.01cm and 0.01cm of 1]  {\Large $a$};
    \node (8) [above = 1cm of 7]  {\Large $d$};
    \node (9) [right = 0.89cm of 7]    {\Large $b$};
    \node (10) [right = 1.10cm of 9]  {\Large $c$};
    \node (11) [above = 1cm of 9]   {\Large $e$};
    \node (12) [above = 1cm of 10]   {\Large $f$};

    \path[draw,thick]
    (1) edge node {} (2)
    (1) edge node {} (3)
    (3) edge node {} (4)
    (3) edge node {} (5)
    (5) edge node {} (6)
    (4) edge node {} (6)
    (2) edge node {} (5)

;
\end{tikzpicture}
\label{FigureSLD}}
\caption{Optimal identifying, self-identifying and self-locating-dominating codes in $G$.} \label{FigureSIDSLDID}

\end{figure}

\begin{definition}
A code $C \subseteq V$ is called \emph{self-identifying} in $G$ if the code $C$ is identifying in $G$ and for all $u \in V$ and $U \subseteq V$ such that $|U| \geq 2$ we have
\[
I(C;u) \neq I(C;U) \text{.}
\]
A self-identifying code $C$ in a finite graph $G$ with the smallest cardinality is called \emph{optimal} and the number of codewords in an optimal self-identifying code is denoted by $\gamma^{SID}(G)$.
\end{definition}

In addition to~\cite{+koodithl}, self-identifying codes have also been previously discussed in~\cite{JLcctl,JLtldsn}. In these papers, two useful characterizations have been presented for self-identifying codes. These characterizations are presented in the following theorem.
\begin{theorem}[\cite{+koodithl,JLcctl,JLtldsn}] \label{ThmSIDChar}
Let $C$ be a code in $G$. Then the following statements are equivalent:
\begin{itemize}
\item[(i)] The code $C$ is self-identifying in $G$.
\item[(ii)] For all distinct $u,v \in V$, we have $I(C;u) \setminus I(C;v) \neq \emptyset$.
\item[(iii)] For all $u \in V$, we have $I(C;u) \neq \emptyset$ and
\[
\bigcap_{c \in I(C;u)} N[c] = \{u\} \text{.}
\]
\end{itemize}
\end{theorem}

The previous definition of self-identifying codes is illustrated in the following example.
\begin{example}
Let $G$ be the graph in Figure~\ref{FigureSID} and let $C$ be a self-identifying code in $G$. If we now have $|I(a)|=1$, then $\bigcap_{c \in I(a)} N[c] \neq \{a\}$ contradicting the fact that $C$ is self-identifying due to Theorem~\ref{ThmSIDChar}(iii). Furthermore, if we have $I(a)=\{b,d\}$, then $I(a)\subseteq I(e)$ which contradicts with Theorem~\ref{ThmSIDChar}(ii). Finally, if $I(a)=\{a,b\}$ or $I(a)=\{a,d\}$, then respectively $I(a)\subseteq I(b)$ or $I(a)\subseteq I(d)$ (a contradiction). Hence, we must have $I(a)=\{a,b,d\}$ if $C$ is self-identifying. Analogously, we get $I(f)=\{c,e,f\}$. Therefore, $C=V$ and, indeed, $V$ is a self-identifying code in $G$.
\end{example}

For the situations when the sensor can distinguish  whether the anomaly is in the open neighbourhood of the sensor or in the location of the sensor itself, we have \emph{locating-dominating} codes which were introduced by Slater in \cite{RS:LDnumber,S:DomLocAcyclic,S:DomandRef} (for recent developments, see \cite{LDBoundExtremal} and \cite{lowww}). More precisely, a code $C \subseteq V$ is locating-dominating in $G$ if the identifying sets $I(C;u)$ are nonempty and unique for all $u \in V \setminus C$.  Inspired by self-identifying codes, we may analogously define so called self-locating-dominating codes, which have been introduced and motivated in~\cite{JLLrntcld}.
\begin{definition}\label{SLD def}
A code $C\subseteq V$ is  \emph{self-locating-dominating} in $G$ if for each vertex $u\in V \setminus C$ we have $I(u) \neq \emptyset$ and
$$
\bigcap_{c\in I(u)} N[c]=\{u\} \text{.}
$$
A self-locating-dominating code $C$ in a finite graph $G$ with the smallest cardinality is called \emph{optimal} and the number of codewords in an optimal self-locating-dominating code is denoted by $\gamma^{SLD}(G)$.
\end{definition}
In the following theorem, we show that self-locating-dominating codes have a characterization analogous to the one of self-identifying codes. By comparing Definition~\ref{SLD def} and Theorem~\ref{SLD karak} to Theorem~\ref{ThmSIDChar}, we can see that they are almost the same except that only non-codewords are considered in the context of self-location-domination.
\begin{theorem}[\cite{JLLrntcld}]\label{SLD karak}
A code $C\subseteq V$ is self-locating-dominating in $G$ if and only if for each vertex $u \in V \setminus C$ and $v \in V$ ($u \neq v$) we have $$I(u)\setminus I(v)\neq\emptyset.$$
\end{theorem}

The previous definition of self-locating-dominating codes is illustrated in the following example.
\begin{example}
Let $G$ be the graph in Figure~\ref{FigureSLD} and let $C$ be a self-locating-dominating code in $G$. Necessarily, the vertex $a$ belongs to $C$ since otherwise
$$
I(a)\setminus I(e)=\emptyset.
$$
Similar reasoning also applies to the vertices $c$, $d$ and $f$. Hence, we have $\{a,c,d,f\}\subseteq C$. On the other hand, we have $N[a]\cap N[c]= \{b\}$ and $N[d] \cap N[f] = \{e\}$. Therefore, by the definition, $C=\{a,c,d,f\}$ is an optimal self-locating-dominating code in $G$.
\end{example}

A graph is called a \emph{complete graph} on $q$ vertices, denoted by $K_q$, if each pair of vertices of the graph is adjacent. The \emph{Cartesian product} of two graphs $G_1=(V_1,E_1)$ and $G_2=(V_2,E_2)$ is defined as $G_1\square G_2=(V_1\times V_2,E)$, where $E$ is a set of edges such that $(u_1,u_2)(v_1,v_2)\in E$ if and only if $u_1=v_1$ and $u_2v_2\in E_2$, or $u_2=v_2$ and $u_1v_1\in E_1$. The Cartesian product $K_q\square K_q\square\cdots \square K_q$ of $n$ copies of $K_q$ is denoted by $K_q^n$.

 Identifying codes have been extensively studied, for example, in the binary hypercubes $K_2^n$ (see the many articles listed in \cite{lowww}), and \cite{HedIDinCarProdPathComp,RWIDinCarProd} for other Cartesian products. In 2008,  Gravier, Moncel and Semri \cite{GravierCompGraph} investigated identification in $K_q^2$.  Goddard and Wash  \cite{GWIDcpg} studied identification in the more general case of $K_q^n$ and they, in particular, gave a conjecture for the cardinality of an optimal identifying code in $K_q^3=K_q\square K_q\square K_q$.

\begin{conjecture}[\cite{GWIDcpg}]\label{Konjektuuri}
For all $q\geq 1$, $\gamma^{ID}(K_q\square K_q\square K_q)=q^2$.
\end{conjecture}

 In~\cite{GWIDcpg}, Goddard and Wash prove that $\gamma^{ID}(K_q^3) \leq q^2$ for all $q \geq 1$. Moreover, by an exhaustive computer search, they show that $\gamma^{ID}(K_q^3)=q^2$ when $q=3$. Furthermore, they provide a lower bound stating that $\gamma^{ID}(K_q^3) \geq q^2-q\sqrt{q}$. Recall that the domination number $\gamma(K_q^3)=\left\lceil\frac{q^2}{2}\right\rceil$ (see~\cite{chll}).

In this paper, we first show a one-to-one correspondence between Latin squares and optimal self-locating-dominating codes in $K_q^3$ in Section~\ref{SectionSLD}. Then based on this observation we see that the bound $q^2$ in Conjecture~\ref{Konjektuuri} holds for \emph{self-locating-dominating codes}. In Section~\ref{SectionID}, we show for identifying codes that $\gamma^{ID}(K_q^3)\leq q^2-q/4$ when $q$ is a power of four. The approach is  based on the recursive use of suitable Latin squares. This result disproves the Conjecture~\ref{Konjektuuri}. We also give constructions of identifying codes for  values of $q$ other than the powers of four.
 After that we improve the lower bound of identifying codes in $K_q^3$ from $q^2-q\sqrt{q}$ to $q^2-\frac{3}{2}q$. Finally, in Section~\ref{SectionIDncopies}, we consider identifying codes in  $K_q^n$, $n>3$. In~\cite{GWIDcpg}, it has been shown that  $\gamma^{ID}(K_q^n)\le q^{n-1}$. 
In Section~\ref{SectionIDncopies}, we significantly improve this upper bound when $q$ is a prime power using suitable linear codes over finite fields as well as self-identifying codes and self-locating-dominating codes.

\section{Self-location-domination in $K_q^3$} \label{SectionSLD}

In this section, we examine  self-locating-dominating codes in $K_q^3$. The vertices of the graph are denoted by $(x,y,z)$, where $1\leq x,y,z\leq q$, i.e., the vertex set $V=\{(x,y,z)\mid (x,y,z)\in\Z^3, 1\leq x,y,z\leq q\}$. Hence, $K^3_q$ can be viewed as a cube in $\Z^3$ consisting of coordinates $(x,y,z)$. A \emph{pipe} is defined as a set of vertices which fixes two of the three coordinates. For example, the set $\{(1,2,z)\mid 1\leq z\leq q\}$ is a pipe in $K^3_q$. A pipe that fixes $y$- and $z$-coordinates is called a \emph{row}, a pipe fixing $x$- and $z$-coordinates is called a \emph{column} and a pipe fixing $x$- and $y$-coordinates is called a \emph{tower}. Two vertices are neighbours in $K^3_q$ if and only if they belong to the same pipe. Hence, we have $|N(v)|=3q-2$.

We can represent a code in $K_q\square K_m\square K_l$ by taking a two dimensional $q\times m$ grid and placing the $z$-coordinates of the codewords to the positions with their $x$- and $y$-coordinates. The tower $(1,1,z)$ is considered to be at the top left corner. Some codes and their representations are illustrated in Example \ref{extEsim}. Moreover, if we choose a suitable code in $K_q^3$, then its representation in a $q\times q$ grid can be considered as a $q\times q$ \textit{Latin square}, which is a $q\times q$ array filled with number from $1$ to $q$ in such a way that each number occurs exactly once in each row or column.

\begin{lemma}\label{3uniikki}
Let $C$ be a code in $K^3_q$ and $v$ be a vertex of $K^3_q$.
\begin{itemize}

\item[(i)] If a vertex $v$ has two codewords in its $I$-set and they do not locate within a single pipe, then there is exactly one other vertex which has those two codewords in its $I$-set.
\item[(ii)] The $I$-set $I(v)$ is not a subset of any other $I$-set if and only if there are at least three codewords in $I(v)$ and they do not locate within a single pipe.

\end{itemize}
\begin{proof}

Let $C$ be a code in $K^3_q$. (i) Let us have $I(u)=\{c_1,c_2\}$, where $u=(x,y,z),c_1=(x_1,y_1,z_1)$ and $c_2=(x_2,y_2,z_2)$. Since $c_1$ and $c_2$ do not belong to the same pipe, we can without loss of generality assume that $x_1\neq x_2$, $y_1\neq y_2$, $x=x_1$, $y=y_2$ and $z=z_1=z_2$. Now we have $N[c_1]\cap N[c_2]=\{u,(x_2,y_1,z)\}$.

(ii) Let us first show that if we have less than three codewords in $I(u)$ or the codewords locate within a single pipe, then $I(u)$ is a subset of another $I$-set. Let us have $I(u)=\{c_1\}$ and $v\in N(c_1)$. Then $I(u)\subseteq I(v)$. If we have $|I(u)|=2$ and the codewords do not locate within the same pipe, then the case is same as in ($i$). If we have $ I(u)=\{c_1,\dots, c_n\}$ and $I(u)\subseteq P$ for  some pipe $P$, then $I(u)\subseteq P\subseteq I(c_1)$. 
Let us then assume that $\{c_1,c_2,c_3\}\in I(u)$, $c_1$, $c_2$ do not belong to the same pipe and $N[c_1]\cap N[c_2]=\{w,u\}$. Hence, $u$ and $w$ do not belong to the same pipe and by (i) we have $N[w]\cap N[u]=\{c_1,c_2\}$. Therefore, $c_3\not\in I(w)$ and $I(u)$ is not a subset of any other $I$-set.\end{proof}
\end{lemma}


In the following two theorems we show that the bound $q^2$ in Conjecture~\ref{Konjektuuri} is true for \emph{self-locating-dominating codes}.
\begin{theorem}\label{SLD alaraja}
We have $$\gamma^{SLD}(K_q^3)\geq q^2.$$
\begin{proof}
Let $C$ be a self-locating-dominating code in $K_q^3$. By Lemma \ref{3uniikki}, each non-codeword has to have at least three codewords in its $I$-set and each codeword has at least one codeword in its $I$-set. Hence, by double counting pairs $(c,x)$ where $c\in C$ and $x\in N[c]$, we get
\[
(3q-2)|C| \geq 3(q^3-|C|)+|C| \iff |C| \geq q^2 \text{.}
\]
\end{proof}
\end{theorem}

In the following theorem, we show with the aid of Lemma~\ref{3uniikki} that each optimal self-locating-dominating code in $K^3_q$ can be represented as a Latin square (and vice versa).
\begin{theorem}
There is a one-to-one correspondence between optimal self-locating-dominating codes in $K^3_q$ and $q\times q$ Latin squares.
\begin{proof}
Let $L$ be a $q\times q$ Latin square. Consider the Latin square $L$ as a code $C$ in $K_q^3$ as in if there is value $z$ in array slot $(x,y)$, then $(x,y,z)\in C$. It is immediate that $|C| = q^2$ as $L$ is a Latin square. Observe that now each non-codeword is covered by exactly three codewords not belonging to a same pipe since there is exactly one codeword in each tower, column and row intersecting with the non-codeword. Hence, by Lemma~\ref{3uniikki}(ii), $C$ is a self-locating-dominating code and we have $\gamma^{SLD}(K_q ^3)=q^2$ due to Theorem \ref{SLD alaraja}.

Let $C$ be an optimal self-locating-dominating code of cardinality $q^2$ in $K^3_q$. By Lemma~\ref{3uniikki}, each non-codeword is covered by at least three codewords and each codeword is covered by at least one codeword. Hence, by double counting pairs $(x,c)$ where $c\in C$ and $c\in N[x]$ we get the inequality $$3q^3-2q^2=(3q-2)|C|\geq |C|+3|V(K^3_q)\setminus C|=3q^3-2q^2$$ and since both sides are equal, each codeword is covered by exactly one codeword (itself) and each non-codeword is covered by exactly three codewords. If there is a tower without codewords, then some tower has two codewords and there is a codeword which is covered by two codewords (a contradiction). So, there is exactly one codeword in each tower. Similarly, we may also show that there is exactly one codeword in each row and column. If we now represent this code using a $q\times q$ grid, we get a Latin square, since there is a number from $1$ to $q$ in each box of the Latin square and the same number never occurs twice in the same row or column.\end{proof}
\end{theorem}

\begin{corollary}
We have $$\gamma^{SLD}(K^3_q)=q^2.$$
\end{corollary}

\section{Constructions for identification in $K_q^3$} \label{SectionID}

In this section, we consider identification in $K^3_q$ and present a bound $\gamma^{ID}(K_q^3)\leq q^2-q/4$ when $q$ is a power of four, giving an infinite family of counterexamples to Conjecture~\ref{Konjektuuri}. First we give a construction for an identifying code in $K_4^3$ with cardinality $15$ and then we use that identifying code and suitable Latin squares to recursively construct the infinite family.



Goddard and Wash~\cite{GWIDcpg} have shown the following results.
\begin{theorem}[\cite{GWIDcpg}]
For $m>2l$ and  $l>2q$, we have
$$\gamma^{ID}(K_q\square K_l\square K_m)= q(m-1) \text{.}$$
Moreover, for the complete graphs of equal order, we have
$$q^2-q\sqrt{q}\leq \gamma^{ID}(K_q^3)\leq q^2.$$
\end{theorem}

Due to the recursive nature of our construction  we first define an operation which combines two codes in $K^3_q$ and $K^3_m$ into a code in $K^3_{qm}$.

\begin{definition}
Let $C_1\subseteq \{(x,y,z)\mid 1\leq x,y,z\leq q\}$ and $C_2\subseteq \{(x,y,z)\mid 1\leq x,y,z\leq m\}$ be codes in $K^3_q$ and $K^3_m$ respectively. Define $\Ext(C_1,C_2)=\{(x,y,z,a,b,c)\mid(x,y,z)\in C_1, (a,b,c)\in C_2\}$.

\end{definition}

The sextuple produced by $\Ext(C_1,C_2)$ (and also other sextuples) can be interpreted in the following way.

\begin{observation}
We can interpret each vertex $(v,u,w)\in K_{qm}^3 = \{(x',y',z')\mid 1\leq x',y',z'\leq qm\}$ as $(x,y,z,a,b,c)$ where $1\leq x,y,z\leq q$ and $1\leq a,b,c\leq m$ by having $v=x+q(a-1)$, $u=y+q(b-1)$ and $w=z+q(c-1)$. Furthermore, since each pipe fixes two out of three coordinates in a triple, a pipe fixes four out of six coordinates in a sextuple.
\end{observation}

\begin{example}\label{extEsim}
Let $C_1=\{(1,1,1),(2,2,2)\}$ and $C_2=\{(1,3,1),(2,2,2),(3,1,3)\}$ be codes in $K^3_2$ and $K^3_3$, respectively (see Figures \ref{K2esim} and \ref{K3esim}). Recall that the vertex $(1,1,z)$ is represented by the top left box and the $z$ coordinate corresponds to the number in that box. Then we have $\Ext(C_1,C_2)=\{(1,1,1,1,3,1),$ $(1,1,1,2,2,2),$ $(1,1,1,3,1,3),$ $(2,2,2,1,3,1),$ $(2,2,2,2,2,2),$ $(2,2,2,3,1,3) \}$. Furthermore, we can consider this as a code in $K^3_6$ as seen in Figure \ref{K6esim}.
\end{example}

\begin{figure}[H]
  \centering
  \begin{minipage}[b]{0.30\textwidth}
  \captionsetup{justification=centering}
    \begin{tikzpicture}
    \matrix[square matrix]
    {
    $1 $ & $ $  \\
    $ $ & $2 $  \\
    };
\end{tikzpicture}
\centering
\caption{Code $C_1$ in \newline  $K_2^3$.}\label{K2esim}   \end{minipage}
  \hfill
  \begin{minipage}[b]{0.30\textwidth}
  \captionsetup{justification=centering}
\begin{tikzpicture}
    \matrix[square matrix]
    {
    $ $ & $ $ & $3 $  \\
    $ $ & $2 $ & $ $  \\
    $1 $ & $ $ & $ $  \\
    };
\end{tikzpicture}
\centering

\caption{ Code $C_2$ in\newline  $K_3^3$.}\label{K3esim}
  \end{minipage}
  \hfill
  \begin{minipage}[b]{0.30\textwidth}
  \captionsetup{justification=centering}
\begin{tikzpicture}
    \matrix[square matrix]
    {
    $ $ & $ $ & $ $ & $ $ & $5 $ & $ $ \\
    $ $ & $ $ & $ $ & $ $ & $ $ & $6 $ \\
    $ $ & $ $ & $3 $ & $ $ & $ $ & $ $ \\
    $ $ & $ $ & $ $ & $4 $ & $ $ & $ $ \\
    $1 $ & $ $ & $ $ & $ $ & $ $ & $ $ \\
    $ $ & $2 $ & $ $ & $ $ & $ $ & $ $ \\
    };
\end{tikzpicture}
\centering

\caption{Code $\Ext(C_1,C_2)$\newline in $K_6^3$.}\label{K6esim}
  \end{minipage}

\end{figure}

Goddard and Wash \cite{GWIDcpg} give the following construction for identifying codes of cardinality $q^2$ in the graph $K^3_q$.
\begin{lemma}\label{n2 koodi}
The code $C_{q}=\{(a,b,c)\mid a+b+c\equiv 0 \pmod q\}$ is identifying in $K^3_q$ with the additional property that each pipe in $K^3_q$ has exactly one codeword.
\begin{proof}
 The code $C_{q}$ is shown to be identifying in \cite{GWIDcpg}. Furthermore, there is exactly one codeword in each pipe since if we fix two of the three coordinates, then for exactly one value of the third coordinate the equation $a+b+c\equiv 0 \pmod q$ is satisfied.
\end{proof}
\end{lemma}

By presenting the previous construction in a grid, we can consider it as a Latin square and hence, we can get the properties mentioned in the previous lemma also that way. The identifying code in $K^3_4$ of the following theorem is of cardinality $15$. The code is presented in Figure~\ref{ID kuutiossa4}.
\begin{theorem}\label{C1}
The code $C^1=\{(2,1,3),$ $(2,1,4),$ $(3,1,1),$ $(4,1,2),$ $(1,2,2),$ $(1,2,4),$ $(2,2,4),$ $(3,2,2),$ $(1,3,1),$ $(2,3,2),$ $(3,3,3),$ $(4,3,3),$ $(2,4,4),$ $(4,4,1),$ $(4,4,3)\}$ is identifying in $K_4^3$.
\begin{proof}
By examining Table~\ref{C^1 table} in the appendix, we notice that each $I$-set is nonempty and unique.
\end{proof}
\end{theorem}

In what follows, we call  the set $Di=\{(j,j,j)\mid 1\leq j\leq 4\}$ as the \emph{diagonal}. We also need another code, $C_L$, to produce the infinite family of codes of the desired cardinality. The code is presented in Figure \ref{ID KpoisD} for the graph $K^3_4[V(K^3_4)\setminus Di]$, that is, for the graph $K^3_4$ with its diagonal vertices $Di$ deleted. Note that the empty squares of the array of Figure~\ref{ID KpoisD} is easy to fill in such a way that a Latin square is obtained.

\begin{lemma}\label{C_L}
The code $C_L=\{(2,1,3)$, $(3,1,4),$ $(4,1,2),$ $(1,2,4),$ $(3,2,1),$ $(4,2,3),$ $(1,3,2),$ $(2,3,4),$ $(4,3,1),$ $(1,4,3),$ $(2,4,1),$ $(3,4,2)\}$ is identifying in $K^3_4[V(K^3_4)\setminus Di]$ and for each codeword $c\in C_L$ we have $I(c)=\{c\}$.
\begin{proof}
By examining Table~\ref{C_L table} in the appendix, we notice that each $I$-set is nonempty and unique. Furthermore, by checking the highlighted vertices we notice that we have $I(c)=\{c\}$ for each codeword $c\in C_L$.
\end{proof}
\end{lemma}

With the help of the codes $C_q,C_L,C^1$ and $Di$, we can construct a family of identifying codes of cardinality $q^2-\frac{q}{4}$ in $K^3_q$ for $q=4^t$, $t\in \Z$ and $t>0$ as described in the following theorem. 

\begin{figure}
  \centering
  \begin{minipage}[b]{0.30\textwidth}
    \captionsetup{justification=raggedright}
    \begin{tikzpicture}
    \clip (-1.21,-1.2) rectangle (2.2,2.45);
    \matrix[bsquare matrix]
    {
    $ $ & $3,4 $ & $1 $ & $2 $ \\
    $2,4 $ & $4 $ & $2 $ & $ $  \\
    $1 $ & $2 $ & $3 $ & $3 $ \\
    $ $ & $4 $ & $ $ & $1,3 $\\
    };
\end{tikzpicture}
\centering
\caption{ Identifying\newline code $C^1$  in $K_4^3$.\newline\newline\newline}\label{ID kuutiossa4}   \end{minipage}
  \hfill
  \begin{minipage}[b]{0.30\textwidth}
    \captionsetup{justification=raggedright}
\begin{tikzpicture}
    \clip (-1.21,-1.2) rectangle (2.2,2.45);
    \matrix[bsquare matrix]
    {
    $ $ & $3 $ & $4 $ & $2 $ \\
    $4 $ & $ $ & $1 $ & $3 $  \\
    $2 $ & $4 $ & $ $ & $1 $ \\
    $3 $ & $1 $ & $2 $ & $ $\\
    };
\end{tikzpicture}
\centering
\caption{ Identifying\newline code $C_L$  in $K_4^3[V(K_4^3)\setminus Di].$\newline\newline\newline}\label{ID KpoisD}  \end{minipage}
  \hfill
  \begin{minipage}[b]{0.31\textwidth}
    \captionsetup{justification=raggedright}
\begin{tikzpicture}
    \clip (-1.21,-1.2) rectangle (2.2,2.45);
    \matrix[bsquare matrix]
    {
    $\mathbf{1} $ & $3 $ & $4 $ & $2 $ \\
    $4 $ & $\mathbf{2}  $ & $1 $ & $3 $  \\
    $2 $ & $4 $ & $\mathbf{3} $ & $1 $ \\
    $3 $ & $1 $ & $2 $ & $\mathbf{4} $\\
    };
\end{tikzpicture}
\centering

\caption{Code $C^t$ where bolded numbers represent cubes with the code $C^{t-1}$\newline and other numbers represent cubes with the code $C_{4^{t-1}}$.}\label{Koodinjakauma4}
  \end{minipage}

\end{figure}

\begin{figure}
\centering
\begin{tikzpicture}
 \draw (-5.27,-5.27) rectangle (5.27,5.27);
  \begin{pgflowlevelscope}{\pgftransformscale{0.75}}
    \matrix[big square matrix]
     {
    |[fill=gray]| $ $ & |[fill=gray]|$3,4 $ & |[fill=gray]|$1 $ &|[fill=gray]| $2 $ & $10 $ & $9 $ & $12 $ & $11 $ & $14 $ & $13 $ & $16 $ & $15 $ & $6 $ & $5 $ & $8 $ & $7 $ \\
   |[fill=gray]| $2,4 $ & |[fill=gray]|$4 $ & |[fill=gray]|$2 $ & |[fill=gray]|$ $ & $9 $ & $12 $ & $11 $ & $10 $ & $13 $ & $16 $ & $15 $ & $14 $ & $5 $ & $8 $ & $7 $ & $6 $ \\
    |[fill=gray]|$1 $ & |[fill=gray]|$2 $ & |[fill=gray]|$3 $ & |[fill=gray]|$3 $ & $12 $ & $11 $ & $10 $ & $9 $ & $16 $ & $15 $ & $14 $ & $13 $ & $8 $ & $7 $ & $6 $ & $5 $ \\
    |[fill=gray]|$ $ & |[fill=gray]|$4 $ & |[fill=gray]|$ $ & |[fill=gray]|$1,3 $ & $11 $ & $10 $ & $9 $ & $12 $ & $15 $ & $14 $ & $13 $ & $16 $ & $7 $ & $6 $ & $5 $ & $8 $ \\
    $ 14$ & $ 13$ & $ 16$ & $ 15$ &  |[fill=gray]|$ $ & |[fill=gray]|$7,8 $ & |[fill=gray]|$ 5$ & |[fill=gray]|$6 $ & $2 $ & $1 $ & $4 $ & $ 3$ & $ 10$ & $ 9$ & $ 12$ & $ 11$ \\
    $ 13$ & $ 16$ & $ 15$ & $14 $ &  |[fill=gray]|$6,8 $ & |[fill=gray]|$ 8$ & |[fill=gray]|$6 $ & |[fill=gray]|$ $ & $1 $ & $4 $ & $3 $ & $2 $ & $ 9$ & $ 12$ & $ 11$ & $ 10$ \\
    $ 16$ & $ 15$ & $ 14$ & $ 13$ &  |[fill=gray]|$5 $ & |[fill=gray]|$ 6$ & |[fill=gray]|$7 $ & |[fill=gray]|$7 $ & $4 $ & $3 $ & $2 $ & $1 $ & $ 12$ & $ 11$ & $ 10$ & $ 9$ \\
    $ 15$ & $ 14$ & $ 13$ & $ 16$ &  |[fill=gray]|$ $ & |[fill=gray]|$ 8$ & |[fill=gray]|$ $ & |[fill=gray]|$ 5,7$ & $3 $ & $2 $ & $1 $ & $ 4$ & $ 11$ & $ 10$ & $ 9$ & $ 12$ \\
    $ 6$ & $ 5$ & $ 8$ & $ 7$ & $ 14$ & $ 13$ & $ 16$ & $ 15$ & |[fill=gray]|$ $ & |[fill=gray]|$ 11,12$ & |[fill=gray]|$ 9$ & |[fill=gray]|$ 10$ & $2 $ & $1 $ & $4 $ & $3 $ \\
    $ 5$ & $ 8$ & $ 7$ & $ 6$ & $ 13$ & $ 16$ & $ 15$ & $ 14$ & |[fill=gray]|$10,12 $ &|[fill=gray]| $12 $ & |[fill=gray]|$10 $ & |[fill=gray]|$ $ & $1 $ & $4 $ & $3 $ & $2 $ \\
    $ 8$ & $ 7$ & $ 6$ & $ 5$ & $ 16$ & $ 15$ & $ 14$ & $ 13$ & |[fill=gray]|$9 $ & |[fill=gray]|$ 10$ & |[fill=gray]|$ 11$ & |[fill=gray]|$ 11$ & $4 $ & $3 $ & $2 $ & $1 $ \\
    $ 7$ & $ 6$ & $ 5$ & $ 8$ & $ 15$ & $ 14$ & $ 13$ & $ 16$ & |[fill=gray]|$ $ & |[fill=gray]|$12 $ & |[fill=gray]|$ $ &|[fill=gray]| $9,11 $ & $3 $ & $2 $ & $1 $ & $ 4$ \\
    $10 $ & $9 $ & $12 $ & $11 $ &  $2 $ & $1 $ & $4 $ & $3 $ & $ 6$ & $ 5$ & $ 8$ & $ 7$ & |[fill=gray]|$ $ & |[fill=gray]|$ 15,16$ & |[fill=gray]|$13 $ & |[fill=gray]|$ 14$ \\
    $ 9$ & $ 12$ & $ 11$ & $ 10$ &  $1 $ & $4 $ & $3 $ & $2 $ & $ 5$ & $ 8$ & $ 7$ & $ 6$ & |[fill=gray]|$ 14,16$ & |[fill=gray]|$ 16$ & |[fill=gray]|$ 14$ & |[fill=gray]|$ $ \\
    $ 12$ & $ 11$ & $ 10$ & $ 9$ &  $4 $ & $3 $ & $2 $ & $1 $ & $ 8$ & $ 7$ & $ 6$ & $ 5$ & |[fill=gray]|$ 13$ & |[fill=gray]|$ 14$ &|[fill=gray]| $ 15$ & |[fill=gray]|$ 15$ \\
    $ 11$ & $ 10$ & $ 9$ & $ 12$ &  $3 $ & $2 $ & $1 $ & $4 $ & $ 7$ & $ 6$ & $ 5$ & $ 8$ & |[fill=gray]|$ $ & |[fill=gray]|$ 16$ & |[fill=gray]|$ $ & |[fill=gray]|$13,15 $ \\
    };

\end{pgflowlevelscope}

\end{tikzpicture}
\caption{An identifying code of size $16^2-4$ in $K^3_{16}$.}\label{16Cube code}
\end{figure}

\begin{theorem}\label{IDkonstruktio}
The code $C^t=\Ext(C_{q/4},C_L)\cup \Ext(C^{t-1},Di)$ is identifying in $K_q^3$, where $q=4^t$ and $t\geq2$, of cardinality $q^2-\frac{q}{4}$.
\begin{proof}
Let $C^t$ be a code recursively defined as $C^t=\Ext(C_{q/4},C_L)\cup \Ext(C^{t-1},Di)$. In other words, the code $C^t$ can be intuitively interpreted as follows. The cube $K_q^3$ can be considered as $K_4^3$ where each vertex is replaced with a \emph{subcube} $K_{q/4}^3$. More precisely, the last three digits of the sextuple notation give the vertex which has been replaced with a $K_{q/4}$ subcube and the first three coordinates give the location within the subcube. Furthermore, the code $C^t$ can be considered as a union of codes $C_{q/4}$ placed into the subcubes given by the code $C_L$ and codes $C^{t-1}$ placed into the subcubes given by the code $Di$ (see Figure~\ref{Koodinjakauma4}). Furthermore, the code $C^2$ is illustrated in Figure~\ref{16Cube code}.

Since codes $C_L$ and $Di$ are separate in the graph $K^3_4$   and $|C^1|=4^2-\frac{4}{4}$, we can use induction on the cardinality $|C^t|$ and thus, have $$|C^t|=|C_L|\cdot |C_{q/4}|+|Di|\cdot |C^{t-1}|=12\cdot\left(\frac{q}{4}\right)^2+4|C^{t-1}|=q^2-\frac{q}{4}.$$ The basic idea of the code $C^t$ is that codes $C_L$ and $Di$  identify the three latter coordinates of the sextuple $(x,y,z,a,b,c)$ and the codes $C_{q/4}$ and $C^{t-1}$ identify the first three coordinates. When $t=2$, $C^1$ is an identifying code in $K^3_4$ (by Theorem~\ref{C1}). Let us now make an induction hypothesis that $C^{t-1}$ is an identifying code in $K^3_{q/4}$ for $t\geq2$.

Let us first show that if for $v=(x,y,z,a,b,c)$ and $w=(x',y',z',a',b',c')$ we have $(a,b,c)\neq (a',b',c')$, then $I(C^t;v)\neq I(C^t;w)$. Let us assume first that $(a,b,c),(a',b',c')\not\in Di$. First notice that by Lemma \ref{n2 koodi} each pipe which goes through a subcube with the code $C_{q/4}$ intersects with exactly one codeword. Hence, if $c=(x'',y'',z'',a'',b'',c'')\in I(C^t;v)$ and $c\not\in \Ext(C^{t-1}, Di)$, then $(a'',b'',c'')\in C_L$. Thus, if $I(C^t;v)=I(C^t;w)$, then $I(C_L;(a,b,c))=I(C_L;(a',b',c'))$ which is not possible since $C_L$ is an identifying code. Moreover, if $(a,b,c)\in Di$, then $I(C^t;v)\subseteq \Ext(C^{t-1},\{(a,b,c)\})$ since $N[Di]\cap C_L=\emptyset$ by Table \ref{n2 koodi} and $K^3_4[Di]$ is a discrete graph.


We now divide the proof into cases based on the location of $v=(x,y,z,a,b,c)$. Let us consider the case where $v$ is such that $(a,b,c)\in C_L$. Moreover, let $w=(x',y',z',a',b',c')$ and let us assume that $I(C^t;v)=I(C^t;w)$. By the previous deduction, we have $(a,b,c)=(a',b',c')$. Moreover, $I(C_{q/4};(x,y,z))=I(C_{q/4};(x',y',z'))$ since $I(C^t;v)=I(C^t;w)$ and $(a,b,c)=(a',b',c')$. Since $C_{q/4}$ is an identifying code, we have $(x,y,z)=(x',y',z')$ and hence, $v=w$.

Let us then consider the case where $(a,b,c)\notin C_L\cup Di$. We have $|I(\Ext(C_{q/4},C_L);v)|\geq 2$ since only codewords have one codeword in their $I$-sets in $C_L$. If  $|I(\Ext(C_{q/4},C_L);v)|\geq 3$, then everything is clear due to Lemma \ref{3uniikki}(ii) since in the codes $C_{q/4}$ and $C_L$ there are no pipes with multiple codewords. Hence, we may assume that  $|I(\Ext(C_{q/4},C_L);v)|=2$. Recall that no pipe in $K^3_q$ has two codewords which are in $\Ext(C_{q/4},C_L)$ and by Lemma \ref{3uniikki}(i) there is exactly one vertex $w$ such that $I(\Ext(C_{q/4},C_L);v)\subseteq I(\Ext(C_{q/4},C_L);w)$. We may assume that the codewords in $I(\Ext(C_{q/4},C_L);v)$ locate in the subcubes which are placed at coordinates $(a',b,c)$ and  $(a,b',c)$. However, now $w$ locates in the subcube at coordinates $(a',b',c)$ which is not possible.

Finally, we have the case $(a,b,c)\in Di$. We immediately notice that if $c\in I(C^t;v)$, then $c$ is of the form $(x',y',z',a,b,c)\in \Ext(C^{t-1},Di)$. Furthermore, the code $C^{t-1}$ is an identifying code in $K_{q/4}^3$ and hence, the vertices within the same diagonal subcube have different $I$-sets than $v$ and thus, the $I$-set of $v$ is unique.
\end{proof}
\end{theorem}

So far, we have given constructions for identifying codes in $K^3_q$ with $q=4^t$. However, we can further use these codes to construct new identifying codes for other values of $q$. For this purpose, we use Latin squares and Evans' Theorem.
\begin{theorem}[Evans' Theorem \cite{evansLatin}]\label{evans}
Any $q\times q$ Latin square can be extended into an $r\times r$ Latin square if $r\geq 2q$.

\end{theorem}

\begin{figure}[H]
\centering
\begin{tikzpicture}
    \clip(-2.3,-2.3) rectangle (2.3,2.3);
    \matrix[square matrix]
    {
     |[fill=gray]| $ $   & |[fill=gray]| $3,4 $ & |[fill=gray]| $1 $ & |[fill=gray]| $2 $ & $9 $ &  $5 $ &  $6 $ &  $7 $ &  $8 $ \\
     |[fill=gray]| $2,4 $ & |[fill=gray]| $4 $   & |[fill=gray]| $2 $ & |[fill=gray]| $ $ & $5 $ &  $6 $ &  $7 $ &  $8 $ &  $9 $  \\
     |[fill=gray]| $1 $   & |[fill=gray]| $2 $   & |[fill=gray]| $3 $ & |[fill=gray]| $3 $ & $6 $ &  $7 $ &  $8 $ &  $9 $ &  $5 $ \\
     |[fill=gray]| $ $   & |[fill=gray]| $4 $   & |[fill=gray]| $ $ & |[fill=gray]| $1,3 $ &  $7 $ &  $8 $ &  $9 $ &  $5 $ &  $6 $ \\
     $ 9$ &  $ 5$ &  $ 6$ &  $ 7$ & $ 8$ & $ 4$ & $ 3$ & $ 2$ & $ 1$ \\
     $ 8$ &  $ 9$ &  $ 5$ &  $ 6$ & $ 4$ & $ 3$ & $ 2$ & $ 1$ & $ 7$ \\
     $ 7$ &  $ 8$ &  $ 9$ &  $ 5$ & $ 3$ & $ 2$ & $ 1$ & $ 6$ & $ 4$ \\
     $ 6$ &  $ 7$ &  $ 8$ &  $ 9$ & $ 2$ & $ 1$ & $ 5$ & $ 4$ & $ 3$ \\
     $ 5$ &  $ 6$ &  $ 7$ &  $ 8$ & $ 1$ & $ 9$ & $ 4$ & $ 3$ & $ 2$ \\
    };
\end{tikzpicture}

\centering
\caption{ Identifying code of size $80$ in $K_9^3$.}\label{9x9 koodi}\centering


\end{figure}

\begin{theorem}\label{laajenna ID}
Let $C$ be an identifying code in $K_q^3$ of cardinality $m$. If $r \geq 2q$, then we have an identifying code in $K_r^3$ of cardinality $r^2-q^2+m$.
\begin{proof}
Let $C$ be an identifying code in $K_q^3$ of cardinality $m$ and $r\geq 2q$. Let us consider a $q\times q$ Latin square $L'$ with values from $1$ to $q$. According to Theorem~\ref{evans}, we can extend the Latin square $L'$ into an $r\times r$ Latin square $L$. Let us assume that the Latin square $L'$ locates in the coordinates $(x,y,z)$, where $x,y,z\leq q$. Moreover, we use notation $(x,y,z)\in L$ if there is value $z$ at the location $(x,y)$ in the Latin square. Let us have
$$
C'=C\cup\left\{(x,y,z)\mid \max\{x,y\}\geq q+1\text{ and } (x,y,z)\in L\right\} \text{.}
$$
The code $C'$ is illustrated in Figure~\ref{9x9 koodi} when $q=4$, $r=9$ and the original code $C = C^1$. We have $|C'|=r^2+m-q^2$. Moreover, we have the following two observations on the structure of the code.

\begin{description}
\item[Observation 1:] Each pipe $P$ with at least one of the two fixed  coordinates greater than $q$, has exactly one codeword in it. Indeed, since $P$ is a pipe with at least one  fixed coordinate  greater than $q$, it does not intersect with $L'$ and hence, it does intersect with $L\setminus L'$. Note that since $L$ is a Latin square, each pipe intersects with exactly one vertex in $L$.
\item[Observation 2:] Vertex $(x,y,z)$ does not belong to $C'$ if exactly one of the three coordinates is greater than $q$. Indeed,  if a vertex $(x,y,z)$ with exactly one coordinate greater than $q$ is a codeword, then the pipe which intersects with $(x,y,z)$ and $L'$ (there is such a pipe) contradicts against the structure of the Latin square $L$.
\end{description}


Let us show that $C'$ is an identifying code by dividing the proof into cases based on the location of the vertex $v = (x,y,z)\in V(K^3_q)$  and whether $v$ is a codeword or not. Let us first consider the case where at least two of the coordinates $(x,y,z)$ of the vertex $v$ are greater than $q$ and $v$ is a non-codeword. Hence, there is exactly one codeword in each pipe intersecting with $v$ by Observation $1$, $|I(v)|=3$ and the codewords in $I(v)$ do not locate within a single pipe. Therefore, $v$ is now uniquely identified by Lemma \ref{3uniikki}.

Let us then consider the case where exactly one of the coordinates $(x,y,z)$ is greater than $q$. Now $v$ is a non-codeword by Observation $2$. We have $|I(v)|\geq 2$ since two of the pipes going through $v$ fix the coordinate which is greater than $q$. Hence, if $|I(v)|>2$, then the $I$-set is unique. On the other hand, if $|I(v)|=2$, then there is another vertex $w$ which has those two codewords in its $I$-set. Now, if $I(v)=\{c,c'\}$, then exactly one coordinate of $c$ is less than $q+1$ and the same is true for $c'$ due to Observation $2$ and since the codewords in $I(v)$ locate in the pipes with a fixed  coordinate greater than $q$.  Moreover, $c$ and $c'$ locate in different pipes and hence, those pipes fix different coordinate as less than $q+1$. Thus, each coordinate of $w$ is greater than $q$. Hence, $|I(w)|\geq3$ and  $I(v)$ is unique.


Let us then consider the case where at least two of the coordinates $(x,y,z)$ are greater than $q$ and $v$ is a codeword. We have $I(v)=\{v\}$ and  each neighbour of $v$ has at least two codewords in its $I$-set as we have seen in the  two previous cases.

Finally, we have the case $x,y,z\leq q$. Now the vertex $v$ is identified by the code $C$ since $C$ is an identifying code, each vertex with a coordinate greater than $q$ has codewords in its $I$-set which do not belong to $C$ by the previous considerations and $I(v)\subseteq C$ by Observation $2$.
\end{proof}
\end{theorem}

By considering the identifying code $C^t$ and Theorem \ref{laajenna ID}, we get the following corollary which gives an identifying code in $K^3_q$ for all $q\geq8$ of cardinality less than $q^2$.
\begin{corollary}
If $2\cdot 4^t\leq q \leq 2\cdot 4^{t+1}-1$, then we have $$\gamma^{ID}(K_q^3)\leq q^2-4^{t-1}.$$
\end{corollary}

\section{Lower bound for identification in $K_q^3$}

With our construction and the lower bound of Goddard and Wash, we now know that $q^2-q\sqrt{q}\leq \gamma^{ID}(K^3_q)\leq q^2-\frac{q}{4}$ when $q=4^t$, $t\in \Z$, $t>0$. In this section,  we improve the lower bound to $\gamma^{ID}(K_q^3)\ge q^2-\frac{3}{2}q$. The standard techniques for obtaining lower bounds for identifying codes in graphs are based on the covering properties of balls or symmetric differences (see \cite{lowww}). For $K_q^3$ these methods are not powerful enough, so we provide a new approach, which builds on the method of Goddard and Wash \cite{GWIDcpg}.

\begin{definition}
Let $C$ be a code in $K_q^3$ and $i$ be an integer such that $1\leq i\le q$. Define an $x_i$\textit{-layer} of the graph $K_q^3$, denoted by $D^1_i$, as the set of vertices which fixes the coordinate $x=i$, i.e., $D^1_i=\{(i,y,z)\mid 1\leq y,z\leq q\}$. Analogously, we define a $y_i$\emph{-layer} $D^2_i$ and a $z_i$\emph{-layer} $D^3_i$. Let $j$ be an integer such that $1\le j\le3$. Define then $X^j_i = \{ v \in D_i^j \mid I(C;v)\cap D^j_i=\emptyset \}$ and $Y_i^j = \{ v \in D_i^j \mid I(C;v)\cap D^j_i= \{ v \} \}$. 
Furthermore, we use the following notation: $X^j=\bigcup_{i=1}^q X_i^j$, $Y^j=\bigcup_{i=1}^q Y_i^j$ and $X=\bigcup_{j=1}^3 X^j$, $ Y=\bigcup_{j=1}^3 Y^j$ and $C^j_i=C\cap D^j_i$. A codeword which does not belong to $Y$ is called a \emph{corner}, and a \emph{fellow} is a codeword belonging to $Y$ such that it has another codeword in its open neighbourhood.
\end{definition}


\begin{lemma}\label{fellowCorner}
Let $C$ be an identifying code in $K^3_q$. For a pipe $P$, the following statements hold:
\begin{itemize}
\item[(i)] The pipe $P$ does not contain multiple fellows.
\item[(ii)] The pipe $P$ does not contain a corner, a fellow and a vertex $v\in X$.
\item[(iii)] The pipe $P$ does not contain a codeword and vertices $v,v'\in X$.
\end{itemize}

\begin{proof}
(i) Let $c$ and $c'$ be fellows in $P$. Therefore, as $c' \in I(c)$ and $c \in I(c')$, we have a contradiction;  $I(c) = I(c')$. 
(ii) Let $c$ be a corner, $c'$ a fellow and $v$ a vertex of $X$ in $P$. Hence, since $c \in I(c')$ and $c \in I(v)$, we have $I(c') \subseteq P$ and $I(v)\subseteq P$. This implies that $I(c)=I(v)$. (iii) Let $c$ be a codeword and $v,v'\in X$ in $P$. Similarly, we again have $I(v)\subseteq P$ and $I(v')\subseteq P$ and, thus, $I(v)=I(v')$. 
\end{proof}
\end{lemma}

We can show that for each vertex $x\in X$, there is a corner that is linked to the vertex $x$. Later, in the proof of Lemma~\ref{kulmavaikute}, we show that each corner is linked to at most three vertices of $X$.
\begin{lemma}\label{kulmatarve}
Let $C$ be an identifying code in $K_q^3$. If $x\in X$, then there exists a codeword $c\in I(x)$ such that $c$ is a corner or a fellow with a corner in $I(c)$.
\begin{proof}
Let $C$ be an identifying code in $K_q^3$ and let $x\in X$. Since $C$ is an identifying code, we have $I(x) \neq \emptyset$. Let us say that $c\in I(x)$. We can assume that $c$ is not a corner since otherwise we are immediately done.  Furthermore, if $I(c)=\{c\}$, then $I(c)=I(x)$. We can now assume that there exists another codeword $c'\in I(c)$ ($c'\neq c$). Now since $c\in Y$ and $|I(c)|\geq 2$, $c$ is a fellow. Moreover, by Lemma~\ref{fellowCorner}(i), $c'$ is not a fellow and therefore it is not in $Y$. Thus, $c'$ is a corner.\end{proof}
\end{lemma}


\begin{definition}
Let $C$ be an identifying code in $K^3_q$ and $c\in C$ be a corner. If there are two codewords $c',c''\in I(c)\cap D^j_i$ which do not belong to the same pipe, then we say that $c$ is a corner of the layer $D^j_i$. Furthermore, we denote by $k^j_i$ the total number of corners of the layer $D^j_i$.
\end{definition}

We have shown that there is corner for each vertex in $X$. We will further show that each corner can be associated to at most three vertices of $X$.
\begin{lemma}\label{kulmavaikute}
Let $C$ be an identifying code in $K_q^3$. Then we have $$|X|\leq 3\sum_{j=1}^3\sum_{i=1}^q k^j_i.$$
\begin{proof}
Let $c\in C$ be a corner.  By Lemma~\ref{fellowCorner}, we have in total at most three fellows and vertices of $X$ in $N[c]$. Moreover, if $c'$ is a fellow, then $|N(c')\cap X|\leq 1$ since if $v,v'\in N(c')\cap X$, then $I(v)=I(v')$. Hence, for each corner $c$ there are at most three vertices in $X$ such that they are in the neighbourhood of $c$ or in the neighbourhood of a fellow $c'\in I(c)$.



With the aid of Lemma \ref{kulmatarve}, we notice that for each vertex $x\in X$ there exists in $N(x)$ a corner or a fellow with a corner in its neighbourhood. Thus, we have $|X|\leq 3\left|\{c\mid c \text{ is a corner}\}\right|$. Moreover, each corner is counted  in the sum $\sum_{j=1}^3\sum_{i=1}^q k^j_i$ and hence, we have $|X|\leq 3\sum_{j=1}^3\sum_{i=1}^q k^j_i$.
\end{proof}
\end{lemma}

To approximate the cardinality of each $X^j_i$, we need the domination number of $K_q \square K_q$. Later, this result is used to approximate the number of vertices of $X$ in a layer. The following lemma is Exercise 1.12 in \cite{DomInGraphs}.
\begin{lemma}[\cite{DomInGraphs}]\label{dominationNumber}
For each positive integer $q$, we have $$\gamma(K_q\square K_q)=q.$$
\end{lemma}

\begin{definition}
Let $C$ be a code in $K^3_q$ and $M^j_i\subseteq D^j_i$ be a minimum dominating set of $D^j_i$ such that $C^j_i\subseteq M^j_i$. Then we denote $f^j_i=|M^j_i|-q$. Note that $q$ is the domination number of $K_q\square K_q$. Let us further denote $a^j_i=q-|C^j_i|$.
\end{definition}


Note that we need $q+f^j_i$ codewords to dominate the layer $D^j_i$. Hence, value $f^j_i$ can be understood as a measurement of how much the structure of the code within a layer increases the cardinality of $X^j_i$. Indeed, observe that if a non-codeword in the layer $D^j_i$ is not dominated by $C^j_i$, then it belongs to $X^j_i$ and there is a row and a column within the layer $D^j_i$ without codewords. Moreover, observe that we have only $q-a^j_i$ codewords in the layer $D^j_i$. Hence, there are at least $q+f^j_i-(q-a^j_i)=f^j_i+a^j_i$ rows and columns which do not have a codeword when $f^j_i+a^j_i\geq0$. Thus, we have $|X^j_i|\ge (a^j_i+f^j_i)^2$. The previous observations are illustrated in Figure \ref{X esim}. Furthermore, the number of corners in a layer is connected with $f^j_i$ as explained in the following lemma.

\begin{lemma}\label{f-k suhde}
Let $C$ be an identifying code in $K^3_q$. Then we have $2f^j_i\geq k^j_i $ for each $i$ and $j$.
\begin{proof}
Let $C$ be an identifying code in $K^3_q$, and consider a layer $D^j_i$ for some $1\leq i\leq q$ and $1\leq j\leq 3$. Since the vertices of $D^j_i$ can be viewed as a graph $K^2_q$, we can consider pipes locating within it as rows and columns. There are $q$ rows and $q$ columns in $D^j_i$ and a subset of $D^j_i$ is dominating if and only if it intersects with all rows or all columns. Indeed, if there are a row $R$ and a column $S$ which do not contain any codewords, then the vertex in their intersection is not dominated. Let us now assume that $C^j_i$ intersects with $n$ rows and $h$ columns ($n\geq h$) and that $M^j_i$ is a minimal dominating set of $D^j_i$ such that $C^j_i\subseteq M^j_i$.

We have $|M^j_i|=|C^j_i|+(q-n)$ since $C^j_i$ dominates $n$ out of $q$ rows and $n\geq h$. Thus, we have $$f^j_i=|M^j_i|-q=|C^j_i|-n.$$ If we have one or two corners of the layer $D^j_i$ in a row, then that row has at least two codewords in it and deleting a corner still preserves codeword in that row. If a row has $m\geq3$ corners in it, then it has at least $m$ codewords and hence, deleting $m-1$ corners still maintains a codeword in the row. Therefore, we may delete at least half of the corners in such a way that there still are codewords in $n$ rows. Hence, we have $n\leq |C^j_i|-\frac{1}{2}k^j_i$ and thus, we have $$f^j_i\geq \frac{1}{2}k^j_i.$$\end{proof}
\end{lemma}

From Lemmas \ref{kulmavaikute} and \ref{f-k suhde} we get following corollary.

\begin{corollary}\label{f-X suhde}
If $C$ is an identifying code in $K^3_q$, then we have $$|X| \leq 6\sum_{j=1}^3\sum_{i=1}^q f^j_i.$$
\end{corollary}

Now we can proof the new lower bound for $\gamma^{ID}(K^3_q)$. The proof is based on the idea that each vertex in $X$ requires corners, corners increases the values $f^j_i$ and the values $f^j_i$ increase the cardinality of $X$. Figure \ref{X esim} shows how corners and the number of codewords affect in the size of $X$.

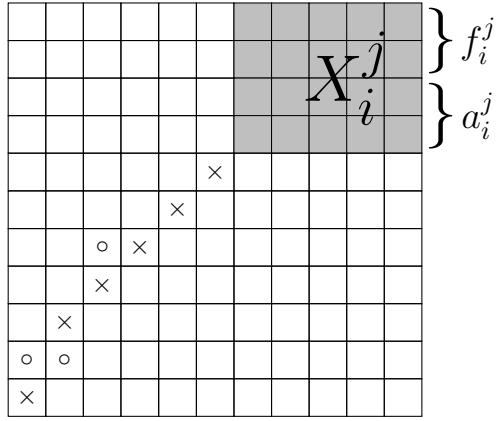
\begin{figure}
\centering
\begin{tikzpicture}
    \matrix[square matrix]
    {
    $ $ & $ $ & $ $ & $ $ & $ $ & $ $ & |[fill=lightgray]|$ $ & |[fill=lightgray]|$ $ & |[fill=lightgray]|$ $ & |[fill=lightgray]|$ $ & |[fill=lightgray]|$ $ \\
    $ $ & $ $ & $ $ & $ $ & $ $ & $ $ & |[fill=lightgray]|$ $ & |[fill=lightgray]|$ $ & |[fill=lightgray]|$ $ & |[fill=lightgray]|$ $ & |[fill=lightgray]|$ $ \\
    $ $ & $ $ & $ $ & $ $ & $ $ & $ $ & |[fill=lightgray]|$ $ & |[fill=lightgray]|$ $ & |[fill=lightgray]|$ $ & |[fill=lightgray]|$ $ & |[fill=lightgray]|$ $ \\
    $ $ & $ $ & $ $ & $ $ & $ $ & $ $ & |[fill=lightgray]|$ $ & |[fill=lightgray]|$ $ & |[fill=lightgray]|$ $ & |[fill=lightgray]|$ $ & |[fill=lightgray]|$ $ \\
    $ $ & $ $ & $ $ & $ $ & $ $ & $\times $ & $ $ & $ $ & $ $ & $ $ & $ $ \\
    $ $ & $ $ & $ $ & $ $ & $\times $ & $ $ & $ $ & $ $ & $ $ & $ $ & $ $  \\
    $ $ & $ $ & $\circ $ & $\times $ & $ $ & $ $ & $ $ & $ $ & $ $ & $ $ & $ $  \\
    $ $ & $ $ & $\times $ & $ $ & $ $ & $ $ & $ $ & $ $ & $ $ & $ $ & $ $ \\
    $ $ & $\times $ & $ $ & $ $ & $ $ & $ $ & $ $ & $ $ & $ $ & $ $ & $ $  \\
    $\circ $ & $\circ $ & $ $ & $ $ & $ $ & $ $ & $ $ & $ $ & $ $ & $ $ & $ $ \\
    $\times $ & $ $ & $ $ & $ $ & $ $ & $ $ & $ $ & $ $ & $ $ & $ $ & $ $  \\
    };

\node[] at (3,1.25) { \Huge $\}$};
\node[] at (3.5,1.25) { \Large $a^j_i$};
\node[] at (3.0,2.25) { \Huge $\}$};
\node[] at (3.5,2.25) { \Large $f^j_i$};
\node[] at (1.75,1.75) { \Huge $X^j_i$};

\end{tikzpicture}
\caption{A code within a layer of $K^3_{12}$ with $3$ corners marked by $\circ$, $7$ other codewords marked by $\times$, $f=2$ and $|X|=20$.}\label{X esim}
\end{figure}
\begin{theorem}

We have $$\gamma^{ID}(K_q^3)\geq q^2-\frac{3}{2} q.$$


\begin{proof}
Let $C$ be an identifying code in $K_q^3$ of an optimal size $\gamma^{ID}(K^3_q)$. 
 We have, \begin{equation}\label{C=q2-a}
|C|=\frac{1}{3}\sum_{j=1}^3\sum_{i=1}^q |C^j_i|=\frac{1}{3}\sum_{j=1}^3\sum_{i=1}^q(q-a^j_i)= q^2-\frac{1}{3}\sum_{j=1}^3\sum_{i=1}^qa^j_i.
\end{equation} Since there exists an identifying code with size $q^2$ by  \cite{GWIDcpg}, we can assume that $\frac{1}{3}\sum_{j=1}^3\sum_{i=1}^qa^j_i\geq 0$. Let $M^j_i$ be a minimum dominating set in $D^j_i$ such that $C^j_i\subseteq M^j_i$. Notice that $f^j_i\geq 0$ and $a^j_i$ can be negative but $a^j_i\geq -f^j_i$ since  $f^j_i=|M^j_i|-q\geq |C^j_i|-q=-a^j_i$. Hence, we have $f^j_i+a^j_i\geq0$.


%

 Now we can give an approximation for $X$: \begin{equation}\label{x>a+f2} |X|\geq\sum_{j=1}^3\sum_{i=1}^q(a^j_i+f^j_i)^2.
\end{equation}
We can do this approximation since there are at least $|M^j_i|-|C^j_i|$ rows and columns without codewords in $D^j_i$. Hence, we have $|X^j_i|\geq \left(|M^j_i|-|C^j_i|\right)^2=\left((q+f^j_i)-(q-a^j_i)\right)^2=\left(f^j_i+a^j_i\right)^2$. 
Furthermore, by Corollary \ref{f-X suhde}, we have \begin{equation}\label{X raja}
6\sum_{j=1}^3\sum_{i=1}^q f^j_i\geq |X|\geq\sum_{j=1}^3\sum_{i=1}^q(a^j_i+f^j_i)^2.
\end{equation}

 Now we can give a lower bound for $|C|$:


\begin{equation*}
\begin{aligned}
|C|& \overset{(\ref{C=q2-a})}{=}  q^2-\frac{1}{3}\sum_{j=1}^3\sum_{i=1}^q a_i^j\\
&=q^2-\frac{1}{3}\left(\sum_{j=1}^3\sum_{i=1}^q a_i^j+\sum_{j=1}^3\sum_{i=1}^qf^j_i-\sum_{j=1}^3\sum_{i=1}^q f_i^j\right)\\
&=q^2-\frac{1}{3}\left(\sum_{j=1}^3\sum_{i=1}^q \left(a_i^j+f^j_i\right)-\sum_{j=1}^3\sum_{i=1}^q f_i^j\right)\\
 & \overset{(\ref{X raja})}{\geq} q^2-\frac{1}{3}\left(\sum_{j=1}^3\sum_{i=1}^q \left(a_i^j+f^j_i\right)-\frac{\sum_{j=1}^3\sum_{i=1}^q \left(a_i^j+f^j_i\right)^2}{6}\right)\\
 & \overset{(*)}{\geq} q^2-\frac{1}{3}\left(3qA-\frac{3qA^2}{6}\right) \\
 & = q^2-q\left(A-\frac{A^2}{6}\right)\\
 & \overset{(**)}{\geq} q^2-\frac{3}{2}q.
\end{aligned}
\end{equation*}

We get the inequality $(*)$ by using Lagrange's method. We can minimize the value of sum $\sum_{j=1}^3\sum_{i=1}^q \left(a_i^j+f^j_i\right)^2$ while retaining the value of sum  $\sum_{j=1}^3\sum_{i=1}^q \left(a_i^j+f^j_i\right)$. The minimum of sum $\sum_{j=1}^3\sum_{i=1}^q \left(a_i^j+f^j_i\right)^2$ equals to $3qA^2$ where $A=\frac{\sum_{j=1}^3\sum_{i=1}^q \left(a_i^j+f^j_i\right)}{3q}$ that is the average value of $a^j_i+f^j_i$. For inequality $(**)$ we notice that $A-\frac{A^2}{6}$ gains its maximum value at $A=3$.\end{proof}
\end{theorem}

\section{Results in $K^n_q$ when $n > 3$} \label{SectionIDncopies}

In this section, we consider identifying, self-identifying and self-locating-dominating codes in  $K_q^n$, when $n > 3$ and $q > 2$.  Goddard and Wash \cite{GWIDcpg}  showed that $\gamma^{ID}(K_q^n)\le q^{n-1}$. In what follows, we first give optimal self-identifying and self-locating-dominating codes in $K_q^n$ for certain values of $n$ and $q$. Then based on these codes we are able to significantly improve the bound $\gamma^{ID}(K_q^n)\le q^{n-1}$ when $q$ is a \emph{prime power}.

For later use, we first begin by introducing some notation and preliminary results based on the classical book~\cite{vanLint} of coding theory. For the rest of the section, we assume that $q$ is a prime power. Then there exists a finite field with $q$ elements, and we denote this field by $\F_q$. The set of all $n$-tuples of $\F_q$ forms a vector space $\F_q^n$. The vector space $\F_q^n$ can be considered as a graph by defining two vertices of $\F_q^n$ to be adjacent if they differ in exactly one coordinate. This graph is called the $q$-ary hypercube or the $q$-ary Hamming space. A vertex of a $q$-ary Hamming space is called a \emph{word}. For two words $u$ and $v$ of $\F_q^n$, the \emph{Hamming distance} is defined as the usual (geodesic) distance $d(u,v)$ of the graph, i.e., the (Hamming) distance is the number of coordinate places in which $u$ and $v$ differ. It is easy to see that the $q$-ary hypercube is isomorphic to the Cartesian product of $n$ copies of $K_q$. Denote the all-zero word of $\F_q^n$ by $\nolla$. A word with one at the $i$th coordinate place and zero in other coordinates is denoted by $e_i$.

A linear subspace of $\F_q^n$ is called a $q$\emph{-ary linear code}. Let $C$ be a linear code in $\F_q^n$ with dimension $d$. Then there exists an $(n-d) \times n$ matrix $H$ such that $Hu^T = \nolla$ if and only if $u \in C$. Now $H$ is called the \emph{parity check matrix} of $C$. Observe that an equivalence relation in $\F_q^n \setminus \{ \nolla \}$ is obtained by defining for all $u,v \in \F_q^n \setminus \{ \nolla \}$, $u \sim v$ if and only if $u = \lambda v$ for some $\lambda \in \F_q^* = \F_q \setminus \{0\}$. Now each equivalence class consists of $q-1$ words of $\F_q^n \setminus \{ \nolla \}$. Assuming $k$ is a positive integer, we form a $k \times (q^k-1)/(q-1)$ matrix $H$ over $\F_q$ by taking as its columns one element from each equivalence class of $\F_q^k \setminus \{ \nolla \}$. Concerning $H$ as a parity check matrix, we obtain a linear code $C$ of length $n$ and dimension $n-k$ such that $|I(C;u)| = 1$ for all $u \in \F_q^n$, i.e., the Hamming distance between any two codewords of $C$ is at least three. The linear code $C$ is called the \emph{Hamming code} of length $n$ and it consists of $q^{n-k}$ codewords.

Let us first begin by presenting a lemma which proves useful in later discussions.
\begin{lemma} \label{LemmaHammingIntersection}
Let $C$ be a code in $\F^n_q$. 
\begin{itemize}
\item[(i)] For two distinct codewords $c_1$ and $c_2$ of $C$, we have
    \[
    |N[c_1] \cap N[c_2]| =
    \left\{
      \begin{array}{ll}
        q, & \hbox{if $d(c_1,c_2) = 1$;} \\
        2, & \hbox{if $d(c_1,c_2) = 2$;} \\
        0, & \hbox{if $d(c_1,c_2) > 2$.}
      \end{array}
    \right.
    \]
\item[(ii)] For three distinct codewords $c_1$, $c_2$ and $c_3$ of $C$ such that there exists a pair of them with the distance equal to $2$ and there exists $u \in \F_q^n$ satisfying $c_1, c_2, c_3 \in N[u]$, we obtain that $N[c_1] \cap N[c_2] \cap N[c_3] = \{u\}$.
\end{itemize}
\end{lemma}
\begin{proof}
(i) Let $c_1, c_2 \in C$ be such that $c_1 \neq c_2$. If $d(c_1,c_2) > 2$, then it is immediate that $N[c_1] \cap N[c_2] = \emptyset$. Furthermore, if $d(c_1, c_2) = 1$, then $c_1$ and $c_2$ differ in exactly one coordinate place and, hence, the intersection $N[c_1] \cap N[c_2]$ consists of all the words having same values in the rest of the $n-1$ coordinate. Therefore, we have $|N[c_1] \cap N[c_2]| = q$. Finally, suppose that $d(c_1, c_2) = 2$. Without loss of generality, we may assume that $c_1 = c_2 + \lambda_1 e_1 + \lambda_2 e_2$ for some $\lambda_i \in \F_q (i=1,2)$. Hence, we have $N[c_1] \cap N[c_2] = \{c_1 + \lambda_1 e_1, c_1 + \lambda_2 e_2\}$ and $|N[c_1] \cap N[c_2]| = 2$.

(ii) Let $c_1$, $c_2$ and $c_3$ be distinct codewords of $C$ such that the distance between two of them is equal to two and there exists $u \in \F_q^n$ satisfying $c_1, c_2, c_3 \in N[u]$. Without loss of generality, we may assume that $d(c_1,c_2)=2$. By the first case, we obtain that $N[c_1] \cap N[c_2] = \{u, v\}$ and $d(u,v) = 2$ for some $v \in \F_q^n$. Now, without loss of generality, we may assume that $c_1 = u + \lambda_1e_1$ and $c_2 = u + \lambda_2e_2$, where $\lambda_1, \lambda_2 \in \F_q$. Therefore, we have $v = u + \lambda_1e_1 + \lambda_2e_2$. Hence, we are immediately done if $u = c_3$ as $d(c_3,v) \geq 2$. Otherwise, $c_3 = u + \lambda'_ie_i$ $(\lambda'_i \in \F_q)$ with $c_3$ being distinct from $c_1$ and $c_2$ and the claim follows as $d(c_3,v) \geq 2$. Thus, in all cases, we obtain that $N[c_1] \cap N[c_2] \cap N[c_3] = \{u\}$.
\end{proof}

In what follows, we introduce an approach to construct identifying codes based on self-identifying codes in $\F_q^n$. We first begin by presenting a characterization for self-identifying codes in $\F_q^n$.
\begin{theorem} \label{ThmHammingSIDChar}
A code $C$ is self-identifying in $\F^n_q$ if and only if for each word $u \in \F^n_q$ we have $|I(C;u)| \geq 3$ and there exist $c_1, c_2 \in I(C;u)$ such that $d(c_1,c_2)=2$.
\end{theorem}
\begin{proof}
Assume that $C$ is a self-identifying code in $\F^n_q$. Suppose first that there exists $u \in \F^n_q$ such that $I(C;u) = \{c_1, c_2\}$, where $c_1, c_2 \in C$. By Lemma~\ref{LemmaHammingIntersection}, we obtain that the intersection $N[c_1] \cap N[c_2]$ contains at least two vertices. This contradicts with the characterization of Theorem~\ref{ThmSIDChar}. Similarly, there does not exist a word $u \in \F^n_q$ which is covered by exactly one codeword of $C$. Hence, we may assume that $I(C;u)$ contains at least $3$ codewords. Suppose then that there does not exist a pair of codewords $c$ and $c'$ in $I(C;u)$ such that $d(c,c')=2$. Hence, the pairwise distance of any two codewords in $I(C;u)$ is one, i.e., the codewords differ in only one coordinate. This implies that the intersection
\[
\bigcap_{c \in I(C;u)} N[c]
\]
contains $q$ words contradicting with the assumption that $C$ is a self-identifying code. Thus, the claim follows.

Assume then that $C \subseteq \F_q^n$ is a code such that for any $u \in \F_q^n$ we have $|I(C;u)| \geq 3$ and there exist $c_1, c_2 \in I(C;u)$ with $d(c_1,c_2)=2$. By Lemma~\ref{LemmaHammingIntersection}(ii), we immediately obtain that for any $v \in \F_q^n$ we have
\[
\bigcap_{c \in I(C;v)} N[c] = \{v \} \text{.}
\]
Thus, $C$ is a self-locating-dominating code in $\F_q^n$.
\end{proof}

For the next theorem, we recall the following notation: for any word $u \in \F_q^n$ and code $C \subseteq \F_q^n$,
\[
u + C = \{u + c \mid c \in C \} \text{.}
\]
In the following theorem, we present an infinite family of optimal self-identifying codes in $\F_q^n$.
\begin{theorem} \label{ThmHammingSIDOptimalConstruction}
Let $q$ be a prime power and let $n$ and $k$ be integers such that $n = (q^k-1)/(q-1)$. If $C$ is a Hamming code in $\F^n_q$, then $C \cup (e_1 + C) \cup (e_2+C)$ is an optimal self-identifying code in $\F^n_q$ with cardinality $3q^{n-k}$.
\end{theorem}
\begin{proof}
Let $C$ be a Hamming code in $\F^n_q$ and denote the code $C \cup (e_1 + C) \cup (e_2+C)$ by $C'$. Since $e_1$, $e_2$ and $e_2-e_1$ do not belong to $C$, the code $C'$ is formed by the Hamming code $C$ and two of its distinct cosets. Hence, it is immediate that each word of $\F_q^n$ is covered by exactly three codewords. Therefore, the claim clearly follows if for all $u \in \F_q^n$ there exist $c_1, c_2 \in I(C';u)$ such that $d(c_1,c_2)=2$. Suppose to the contrary that there exists a word $v \in \F_q^n$ such that the pairwise distance of any two codewords of $I(C;v)$ is one. By the construction of $C'$, we have $I(C';v) = \{c_1, c_2, c_3\}$, where $c_1 \in C$, $c_2 \in e_1 + C$ and $c_3 \in e_2 + C$. Therefore, as $d(c_1, c_2) = 1$ and $d(c_1, c_3) = 1$, we respectively have $c_2 = c_1 + e_1$ and $c_3 = c_1 + e_2$. However, then a contradiction follows since $d(c_2, c_3) = 2$. Thus, the code $C'$ is self-identifying in $\F_q^n$ with $|C'| = 3q^{n-k}$ (as $|C| = q^{n-k}$).

On the other hand, if $D$ is an arbitrary self-identifying code in $\F_q^n$, then each word of $\F_q^n$ is covered by at least three codewords of $D$ by Theorem~\ref{ThmHammingSIDChar}. Therefore, using a double counting argument similar to the proof of Theorem~\ref{SLD alaraja}, we obtain that $(n(q-1)+1)|D| \geq 3q^n$ which further implies $|D| \geq 3q^{n-k}$. Therefore, the self-identifying code $C'$ in $\F_q^n$ is optimal as $|C'| = 3q^{n-k}$. This concludes the proof of the claim.
\end{proof}

For the following theorem, we recall the notation of the \emph{direct sum}: for any codes $C_1 \subseteq \F_q^{n_1}$ and $C_2 \subseteq \F_q^{n_2}$, where $n_1$ and $n_2$ are positive integers, we denote
\[
C_1 \oplus C_2 = \{ (c_1, c_2) \mid c_1 \in C_1, c_2 \in C_2 \} \text{.}
\]
In the following theorem, we present a simple method of constructing self-identifying codes in $\F_q^{n+1}$ based on the ones in $\F_q^{n}$.
\begin{theorem}
If $C$ is a self-identifying code in $\F^n_q$, then $C \oplus \F_q$ is a self-identifying code in $\F^{n+1}_q$.
\end{theorem}
\begin{proof}
Let $C$ be a self-identifying code in $\F^n_q$ and $u = (u_1, u_2)$ be a word of $\F_q^{n+1}$ such that $u_1 \in \F_q^n$ and $u_2 \in \F_q$. By Theorem~\ref{ThmHammingSIDChar}, the word $u_1$ is covered by at least three codewords $c_1$, $c_2$ and $c_3$ of $C$ with the additional property that the distance between two of them is equal to two. Hence, the word $u \in \F_q^{n+1}$ is covered at least by the codewords $(c_1, u_2)$, $(c_2, u_2)$ and $(c_3, u_2)$ of $C \oplus \F_q$, and the codewords satisfy the additional property that the distance between two of them is equal to two. Therefore, by Theorem~\ref{ThmHammingSIDChar}, the code $C \oplus \F_q$ is self-identifying in $\F^{n+1}_q$.
\end{proof}

Recall that each self-identifying code is always identifying. Therefore, by the previous theorems, we have $\gamma^{ID}(\F_q^n) \leq 3q^{n-k}$ for integers $n$, $k$ and $\ell$ such that $n = (q^k-1)/(q-1) + \ell$, where $q$ is a prime power. This significantly improves over the previous upper bound $\gamma^{ID}(\F_q^n) \leq q^{n-1}$ by Goddard and Wash~\cite{GWIDcpg}; however, recall that they do not require that $q$ is a prime power. In what follows, we introduce another way to construct identifying codes based on self-locating-dominating codes in $\F_q^n$. We first begin by presenting a characterization for self-locating-dominating codes in $\F_q^n$.
\begin{theorem} \label{ThmHammingSLDChar}
A code $C$ is self-locating-dominating in $\F^n_q$ if and only if for each word $u \in \F^n_q \setminus C$ we have $|I(C;u)| \geq 3$ and there exist $c_1, c_2 \in I(C;u)$ such that $d(c_1,c_2)=2$.
\end{theorem}
\begin{proof}
Recall that a code $C$ is self-locating-dominating if for all $u \in \F_q^n \setminus C$
\[
\bigcap_{c \in I(C;u)} N[c] = \{u \} \text{.}
\]
By Theorem~\ref{ThmSIDChar}, a code is self-identifying if and only if the same condition is satisfied for all words $u \in \F_q^n$. Hence, the claim follows by an argument analogous to the proof of Theorem~\ref{ThmHammingSIDChar}.
\end{proof}

In the following theorem, we present an infinite family of optimal self-locating-dominating codes in $\F_q^n$.
\begin{theorem} \label{ThmHammingSLDOptimalConstruction}
Let $q$ be a prime power and let $n$ and $k$ be integers such that $n = 3(q^k-1)/(q-1)$. Assume that $H$ is a $k \times n$ parity check matrix formed from the $k \times (n/3)$ parity check matrix of the Hamming code by repeating each column three times. Now the code $C$ corresponding to the parity check matrix $H$ is an optimal self-locating-dominating code in $\F^n_q$ with cardinality $q^{n-k}$.
\end{theorem}
\begin{proof}
Let $u$ be a word of $\F_q^n$. Now we obtain the following observations:
\begin{itemize}
\item Suppose that $Hu^T = x \in \F_q^k$ and $x \neq \nolla$. Due to the construction of the parity check matrix $H$, there exist exactly three columns $h_{i_1}$, $h_{i_2}$ and $h_{i_3}$ of $H$ such that $x = \lambda h_{i_1} = \lambda h_{i_2} = \lambda h_{i_3}$ for some $\lambda \in \F_q$. Hence, there exist exactly three words $\lambda e_{i_1}$, $\lambda e_{i_2}$ and $\lambda e_{i_3}$ of weight one in $\F_q^n$ such that the indices $i_j$ are all different and $H(u+\lambda e_{i_j})^T = \nolla$, i.e., $u + \lambda e_{i_j}$ belongs to $C$. Therefore, the word $u$ is covered by exactly three codewords of $C$ in $\F_q^n$. Moreover, the distance between any of these codewords is equal to two.
\item If $Hu^T = \nolla \in \F_q^k$, then analogously to the previous case we can observe that $u \in C$ is covered by exactly one codeword of $C$ in $\F_q^n$; namely, by itself.
\end{itemize}
Thus, by the previous observations, we know that $I(C;u) = \{u \}$ if $u \in C$ and for non-codewords $u \in \F_q^n \setminus C$ we have $|I(C;u)|=3$ with the additional property that the distance between any two codewords of $I(C;u)$ is equal to two. Therefore, by Theorem~\ref{ThmHammingSLDChar}, the code $C$ is self-locating-dominating in $\F_q^n$. Moreover, it is easy to calculate that $|C| = q^{n-k}$.

On the other hand, if $D$ is an arbitrary self-locating-dominating code in $\F_q^n$, then each word of $\F_q^n \setminus D$ is covered by at least three codewords of $D$ by Theorem~\ref{ThmHammingSLDChar}. Therefore, using a double counting argument similar to the proof of Theorem~\ref{ThmHammingSIDOptimalConstruction}, we obtain that $(n(q-1)+1)|D| \geq 3(q^n-|D|) + |D|$ which further implies $|D| \geq q^{n-k}$. Therefore, the self-locating-dominating code $C$ in $\F_q^n$ is optimal as $|C| = q^{n-k}$. This concludes the proof of the claim.
\end{proof}

In a similar way, we can also construct self-locating-dominating codes (albeit not optimal) for other lengths $n$.
\begin{theorem} \label{ThmHammingSLDGeneralConstruction}
Let $q$ be a prime power and let $n$, $k$ and $\ell$ be integers such that $n = 3(q^k-1)/(q-1) + \ell$. Assume that $H$ is a $k \times n$ parity check matrix formed from the $k \times ((n-\ell)/3)$ parity check matrix of the Hamming code by repeating the first column $\ell +3$ times and each other column three times. Now the code $C$ corresponding to the parity check matrix $H$ is self-locating-dominating in $\F^n_q$ with cardinality $q^{n-k}$.
\end{theorem}
\begin{proof}
The proof of the claim is similar to the one of Theorem~\ref{ThmHammingSLDOptimalConstruction}.
\end{proof}

Observe that the codes $C \subseteq \F^n_q$ constructed in Theorems~\ref{ThmHammingSLDOptimalConstruction} and \ref{ThmHammingSLDGeneralConstruction} are such that for each codeword $c \in V$ we have $I(c) = \{c\}$ and for non-codewords $u \in \F^n_q \setminus C$ we have
\[
\bigcap_{c \in I(u)} N[c] = \{u \} \text{.}
\]
Therefore, all the constructed self-locating-dominating codes are also identifying in $\F^n_q$. Hence, we have $\gamma^{ID}(\F_q^n) \leq q^{n-k}$ for all integers $n$, $k$ and $\ell$ such that $n = 3(q^k-1)/(q-1) + \ell$, where $q$ is a prime power. Thus, using the constructions based on the self-locating-dominating codes, we are able to significantly improve the previous upper bound $\gamma^{ID}(\F_q^n) \leq q^{n-1}$ by~\cite{GWIDcpg} (recall again that in~\cite{GWIDcpg} it is not required that $q$ is a prime power).

In~\cite{GWIDcpg}, it is stated that the best known lower bound for identifying codes in $\F_q^n$ is the following one by Karpovsky \emph{et al.}~\cite{kcl}.
\begin{theorem}[\cite{kcl}]
We have
\[
\gamma^{ID}(\F_q^n) \geq \frac{2q^n}{nq-n+2} \text{.}
\]
\end{theorem}
Assume that $q$ is a prime power and $n$ and $k$ are integers such that $n = 3(q^k-1)/(q-1)$. As stated above, we now have $\gamma^{ID}(\F_q^n) \leq q^{n-k}$. Now the previous lower bound can be written as follows:
\[
\gamma^{ID}(\F_q^n) \geq \frac{2q^n}{nq-n+2} = \frac{2q^n}{3q^k-1} \geq \frac{2q^n}{3q^k} = \frac{2}{3} q^{n-k} \text{.}
\]
Hence, comparing the previous lower and upper bounds, it can be seen that they are of the same order $\Theta(q^{n-k})$. Analogously, it can be shown that the (self-)identifying codes obtained in Theorem~\ref{ThmHammingSIDOptimalConstruction} for lengths $n = (q^k-1)/(q-1)$ are also rather small compared to the lower bound above.

\bibliographystyle{abbrv}

\begin{thebibliography}{10}

\bibitem{IDHereditary}
N.~Bousquet, A.~Lagoutte, Z.~Li, A.~Parreau and S.~Thomass\'e.
\newblock {\em Identifying codes in hereditary classes of graphs and {VC}-dimension}.
\newblock {\em SIAM J. Discrete Math.}, 29(4):2047--2064, 2015.

\bibitem{LDBoundExtremal}
J.~C\'aceres, C.~Hernando, M.~Mora, I.~M.~Pelayo and M.~L.~Puertas.
\newblock {\em Locating-dominating codes: bounds and extremal cardinalities}.
\newblock {\em Applied Mathematics and Computation}, 220:38--45, 2013.

\bibitem{chll}
G.~Cohen, I.~Honkala, S.~Litsyn and A.~Lobstein.
\newblock {\em Covering codes}, volume~54 of {\em North-Holland Mathematical Library}.
\newblock North-Holland Publishing Co., Amsterdam, 1997.

\bibitem{GWIDcpg}
W.~Goddard and K.~Wash.
\newblock I{D} codes in {C}artesian products of cliques.
\newblock {\em J. Combin. Math. Combin. Comput.}, 85:97--106, 2013.

\bibitem{GravierCompGraph}
S.~Gravier, J.~Moncel, and A.~Semri.
\newblock Identifying codes of {C}artesian product of two cliques of the same
  size.
\newblock {\em Electron. J. Combin.}, 15(1):Note 4, 7, 2008.

\bibitem{evansLatin}
T.~Evans.
\newblock Embedding Incomplete Latin Squares.
\newblock {\em The American Mathematical Monthly}, number 10, volume 67, pages 958--961, 1960.

\bibitem{Trachtenberg}
N.~Fazlollahi, D.~Starobinski, and A.~Trachtenberg.
\newblock Connected identifying codes.
\newblock {\em IEEE Trans. Inform. Theory}, 58(7):4814--4824, 2012.

\bibitem{IDvTrans}
S.~Gravier, A.~Parreau, S.~Rottey, L.~Storme and \'E.~Vandomme.
\newblock {\em Identifying codes in vertex-transitive graphs and strongly regular graphs}.
\newblock {\em Electron. J. Combin.}, 22(4):Paper 4.6, 26 pp, 2015.

\bibitem{DomInGraphs}
T.~Haynes, S.~Hedetniemi and P.~Slater.
\newblock {\em Fundamentals of Domination in Graphs}, volume~208 of {\em Pure and Applied Mathematics}.
\newblock Marcel Dekker, Inc., New York, 1998.

\bibitem{HedIDinCarProdPathComp}
S.~Hedetniemi.
\newblock On identifying codes in the Cartesian product of a path and a complete graph.
\newblock {\em Journal of Combinatorial Optimization}, 31(4):1405--1416, 2016.

\bibitem{+koodithl}
I.~Honkala and T.~Laihonen.
\newblock On a new class of identifying codes in graphs.
\newblock {\em Inform. Process. Lett.}, 102(2-3):92--98, 2007.

\bibitem{IDMoreCompl}
O.~Hudry and A.~Lobstein.
\newblock More results on the complexity of identifying problems in
              graphs.
\newblock {\em Theoret. Comput. Sci.}, 626:1--12, 2016.


\bibitem{JLcctl}
V.~Junnila and T.~Laihonen.
\newblock Collection of codes for tolerant location.
\newblock In {\em Proceedings of the Bordeaux Graph Workshop}, pages 176--179,
  2016.



\bibitem{JLtldsn}
V.~Junnila and T.~Laihonen.
\newblock Tolerant location detection in sensor networks.
\newblock Submitted, 2016.

\bibitem{JLLICGT18}
V.~Junnila, T.~Laihonen, and T.~Lehtil{\"a}.
\newblock On a Conjecture regarding Identification in the Hamming Graphs.
\newblock {\em 10th International Colloquium on Graph Theory and combinatorics, July 9-13, 2018}, accepted.

\bibitem{JLLrntcld}
V.~Junnila, T.~Laihonen, and T.~Lehtil{\"a}.
\newblock On regular and new types of codes for location-domination.
\newblock {\em Discrete Applied Mathematics}, to appear, 2018.
\newblock https://doi.org/10.1016/j.dam.2018.03.050

\bibitem{kcl}
M.~G. Karpovsky, K.~Chakrabarty, and L.~B. Levitin.
\newblock On a new class of codes for identifying vertices in graphs.
\newblock {\em IEEE Trans. Inform. Theory}, 44(2):599--611, 1998.

\bibitem{LT:disj}
M.~Laifenfeld and A.~Trachtenberg.
\newblock Disjoint identifying-codes for arbitrary graphs.
\newblock In {\em Proceedings of International Symposium on Information Theory,
  2005. ISIT 2005}, pages 244--248, 2005.

\bibitem{vanLint}
J. H. van Lint.
\newblock {\em Introduction to coding theory}, volume~86 of {\em Graduate Texts in Mathematics}.
\newblock Springer-Verlag, Berlin, 1999.

\bibitem{lowww}
A.~Lobstein.
\newblock Watching systems, identifying, locating-dominating and discriminating codes in graphs, a bibliography.
\newblock Published electronically at {\tt
  http://perso.enst.fr/$\sim$lobstein/debutBIBidetlocdom.pdf}.


\bibitem{RS:LDnumber}
D.~F. Rall and P.~J. Slater.
\newblock On location-domination numbers for certain classes of graphs.
\newblock {\em Congr. Numer.}, 45:97--106, 1984.

\bibitem{RWIDinCarProd}
D.~F. Rall and K.~Wash.
\newblock On Minimum Identifying Codes in Some Cartesian Product Graphs.
\newblock {\em Graphs and Combinatorics}, 33(4):1037--1053, 2017.

\bibitem{Ray}
S.~Ray, D.~Starobinski, A.~Trachtenberg, and R.~Ungrangsi.
\newblock Robust location detection with sensor networks.
\newblock {\em IEEE Journal on Selected Areas in Communications},
  22(6):1016--1025, August 2004.

\bibitem{S:DomLocAcyclic}
P.~J. Slater.
\newblock Domination and location in acyclic graphs.
\newblock {\em Networks}, 17(1):55--64, 1987.

\bibitem{S:DomandRef}
P.~J. Slater.
\newblock Dominating and reference sets in a graph.
\newblock {\em J. Math. Phys. Sci.}, 22:445--455, 1988.




\end{thebibliography}

\section*{Appendix}
 Below we give the tables which show that the code $C^1$ in Theorem \ref{C1} and the code $C_L$ in Theorem \ref{C_L} are identifying codes.

\begin{table}
\centering
{\small
\begin{tabular}{ |c|c|c|c| }
\hline
$I(1,1,1)$ & $\{(3,1,1),(1,3,1)\}$ & $I(2,1,1)$ & $\{(3,1,1),(2,1,3),(2,1,4)\}$  \\
 \hline
 $I(3,1,1)$ & $\{(3,1,1) \}$  &  $I(4,1,1)$ & $\{(3,1,1),(4,4,1),(4,1,2)\}$   \\
 \hline
$I(1,2,1)$ & $\{(1,3,1),(1,2,2),(1,2,4)\}$  & $I(2,2,1)$ & $\{(2,2,4)\}$  \\
 \hline
$I(3,2,1)$ & $\{(3,1,1),(3,2,2)\}$  & $I(4,2,1)$ & $\{(4,4,1)\}$  \\
 \hline
$I(1,3,1)$ & $\{(1,3,1)\}$  & $I(2,3,1)$ & $\{(1,3,1),(2,3,2)\}$  \\
 \hline
$I(3,3,1)$ & $\{(1,3,1),(3,3,3),(3,1,1)\}$  & $I(4,3,1)$ & $\{(1,3,1),(4,3,3),(4,4,1)\}$  \\
 \hline
$I(1,4,1)$ & $\{(1,3,1),(4,4,1)\}$  & $I(2,4,1)$ & $\{(2,4,4),(4,4,1)\}$  \\
 \hline
$I(3,4,1)$ & $\{(3,1,1),(4,4,1)\}$  & $I(4,4,1)$ & $\{(4,4,1),(4,4,3)\}$  \\
 \hline
$I(1,1,2)$ & $\{(1,2,2),(4,1,2)\}$  & $I(2,1,2)$ & $\{(2,1,3),(2,1,4),(2,3,2),(4,1,2)\}$  \\
 \hline
$I(3,1,2)$ &  $\{(3,1,1),(3,2,2),(4,1,2)\}$  & $I(4,1,2)$ & $\{(4,1,2)\}$  \\
 \hline
$I(1,2,2)$ & $\{(1,2,2),(1,2,4),(3,2,2)\}$  & $I(2,2,2)$ & $\{(1,2,2),(2,2,4),(2,3,2),(3,2,2)\}$  \\
 \hline
$I(3,2,2)$ & $\{(1,2,2),(3,2,2)\}$  & $I(4,2,2)$ & $\{(1,2,2),(3,2,2),(4,1,2)\}$  \\
 \hline
$I(1,3,2)$ & $\{(1,2,2),(1,3,1),(2,3,2)\}$  & $I(2,3,2)$ & $\{(2,3,2)\}$  \\
 \hline
$I(3,3,2)$ & $\{(2,3,2),(3,2,2),(3,3,3)\}$  & $I(4,3,2)$ & $\{(2,3,2),(4,1,2),(4,3,3)\}$  \\
 \hline
$I(1,4,2)$ & $\{(1,2,2)\}$  & $I(2,4,2)$  & $\{(2,3,2),(2,4,4)\}$  \\
 \hline
$I(3,4,2)$ & $\{(3,2,2)\}$  & $I(4,4,2)$ & $\{(4,1,2),(4,4,1),(4,4,3)\}$  \\
 \hline
$I(1,1,3)$ & $\{(2,1,3)\}$  & $I(2,1,3)$ & $\{(2,1,3),(2,1,4)\}$  \\
\hline
$I(3,1,3)$ & $\{(2,1,3),(3,1,1),(3,3,3)\}$  & $I(4,1,3)$ & $\{(2,1,3),(4,1,2),(4,3,3),(4,4,3)\}$  \\
 \hline
$I(1,2,3)$ & $\{(1,2,2),(1,2,4)\}$  & $I(2,2,3)$ & $\{(2,1,3),(2,2,4)\}$  \\
 \hline
$I(3,2,3)$ & $\{(3,2,2),(3,3,3)\}$  & $I(4,2,3)$ & $\{(4,3,3),(4,4,3)\}$  \\
 \hline
$I(1,3,3)$ & $\{(1,3,1),(3,3,3),(4,3,3)\}$  & $I(2,3,3)$ & $\{(2,1,3),(2,3,2),(3,3,3),(4,3,3)\}$  \\
\hline
$I(3,3,3)$ & $\{(3,3,3),(4,3,3)\}$  & $I(4,3,3)$ & $\{(3,3,3),(4,3,3),(4,4,3)\}$  \\
 \hline
$I(1,4,3)$ & $\{(4,4,3)\}$  & $I(2,4,3)$ & $\{(2,1,3),(2,4,4),(4,4,3)\}$  \\
 \hline
$I(3,4,3)$ & $\{(3,3,3),(4,4,3)\}$  & $I(4,4,3)$ & $\{(4,3,3),(4,4,1),(4,4,3)\}$  \\
 \hline
$I(1,1,4)$ & $\{(1,2,4),(2,1,4)\}$  & $I(2,1,4)$ & $\{(2,1,3),(2,1,4),(2,2,4),(2,4,4)\}$  \\
\hline
$I(3,1,4)$ & $\{(2,1,4),(3,1,1)\}$  & $I(4,1,4)$ & $\{(2,1,4),(4,1,2)\}$  \\
 \hline
$I(1,2,4)$ & $\{(1,2,2),(1,2,4),(2,2,4)\}$  & $I(2,2,4)$ & $\{(1,2,4),(2,1,4),(2,2,4),(2,4,4)\}$  \\
 \hline
$I(3,2,4)$ & $\{(1,2,4),(2,2,4),(3,2,2)\}$  & $I(4,2,4)$ & $\{(1,2,4),(2,2,4)\}$  \\
 \hline
$I(1,3,4)$ & $\{(1,2,4),(1,3,1)\}$  & $I(2,3,4)$ & $\{(2,1,4),(2,2,4),(2,3,2),(2,4,4)\}$  \\
\hline
$I(3,3,4)$ & $\{(3,3,3)\}$  & $I(4,3,4)$ & $\{(4,3,3)\}$  \\
 \hline
$I(1,4,4)$ & $\{(1,2,4),(2,4,4)\}$  & $I(2,4,4)$ & $\{(2,1,4),(2,2,4),(2,4,4)\}$  \\
 \hline
$I(3,4,4)$ & $\{(2,4,4)\}$  & $I(4,4,4)$ & $\{(2,4,4),(4,4,1),(4,4,3)\}$  \\
 \hline

\end{tabular}
}

 \caption{$I$-sets of code $C^1$.}\label{C^1 table}

\end{table}

\begin{table}
\centering
{\small
\begin{tabular}{ |c|c|c|c| }
\hline
$I(1,1,1)$ & $\emptyset$ & $I(2,1,1)$ &  $\{ (2,1,3), (2,4,1) \}$  \\
 \hline
$I(3,1,1)$ & $\{(3,1,4),(3,2,1) \}$  & $I(4,1,1)$ & $\{(4,1,2), (4,3,1)\}$  \\
 \hline
$I(1,2,1)$ & $\{(1,2,4), (3,2,1) \}$  & $I(2,2,1)$ & $\{(2,4,1), (3,2,1)  \}$  \\
 \hline
\cellcolor{lightgray}$I(3,2,1)$ & \cellcolor{lightgray}$\{(3,2,1) \}$  & $I(4,2,1)$ & $\{(3,2,1),(4,2,3),(4,3,1)  \}$  \\
 \hline
$I(1,3,1)$ & $\{(1,3,2),(4,3,1) \}$  & $I(2,3,1)$ & $\{(2,3,4), (2,4,1),(4,3,1) \}$  \\
 \hline
$I(3,3,1)$ & $\{(3,2,1),(4,3,1) \}$  & \cellcolor{lightgray}$I(4,3,1)$ & \cellcolor{lightgray}$\{(4,3,1)\}$  \\
 \hline
$I(1,4,1)$ & $\{(1,4,3), (2,4,1) \}$  & \cellcolor{lightgray}$I(2,4,1)$ &\cellcolor{lightgray} $\{ (2,4,1) \}$  \\
 \hline
$I(3,4,1)$ & $\{(2,4,1), (3,2,1), (3,4,2)  \}$  & $I(4,4,1)$ & $\{(2,4,1),(4,3,1) \}$  \\
 \hline
$I(1,1,2)$ & $\{(1,3,2),(4,1,2)  \}$  & $I(2,1,2)$ & $\{ (2,1,3), (4,1,2)  \}$  \\
 \hline
$I(3,1,2)$ & $\{(3,1,4),(3,4,2),(4,1,2)   \}$  &\cellcolor{lightgray} $I(4,1,2)$ &\cellcolor{lightgray} $\{(4,1,2)\}$  \\
 \hline
$I(1,2,2)$ & $\{(1,2,4),(1,3,2)  \}$  & $I(2,2,2)$ & $\emptyset$  \\
 \hline
$I(3,2,2)$ & $\{(3,2,1), (3,4,2) \}$  & $I(4,2,2)$ & $\{(4,1,2),(4,2,3)  \}$  \\
 \hline
\cellcolor{lightgray}$I(1,3,2)$ &\cellcolor{lightgray} $\{(1,3,2) \}$  & $I(2,3,2)$ & $\{(1,3,2), (2,3,4) \}$  \\
 \hline
$I(3,3,2)$ & $\{(1,3,2), (3,4,2) \}$  & $I(4,3,2)$ & $\{(1,3,2),(4,1,2), (4,3,1) \}$  \\
 \hline
$I(1,4,2)$ & $\{(1,3,2), (1,4,3),(3,4,2) \}$  & $I(2,4,2)$ & $\{(2,4,1),(3,4,2)  \}$  \\
 \hline
\cellcolor{lightgray}$I(3,4,2)$ &\cellcolor{lightgray} $\{(3,4,2)\}$  & $I(4,4,2)$ & $\{(3,4,2),(4,1,2)  \}$  \\
 \hline
$I(1,1,3)$ & $\{(1,4,3), (2,1,3) \}$  &\cellcolor{lightgray} $I(2,1,3)$ &\cellcolor{lightgray} $\{ (2,1,3) \}$  \\
\hline
$I(3,1,3)$ & $\{(2,1,3), (3,1,4)  \}$  & $I(4,1,3)$ & $\{ (2,1,3),(4,1,2),(4,2,3) \}$  \\
 \hline
$I(1,2,3)$ & $\{(1,2,4), (1,4,3),(4,2,3) \}$  & $I(2,2,3)$ & $\{ (2,1,3),(4,2,3)  \}$  \\
 \hline
$I(3,2,3)$ & $\{(3,2,1),(4,2,3)  \}$  &\cellcolor{lightgray} $I(4,2,3)$ & \cellcolor{lightgray}$\{(4,2,3)\}$  \\
 \hline
$I(1,3,3)$ & $\{(1,3,2), (1,4,3)\}$  & $I(2,3,3)$ & $\{ (2,1,3),(2,3,4)  \}$  \\
\hline
$I(3,3,3)$ & $\emptyset$  & $I(4,3,3)$ & $\{(4,2,3),(4,3,1) \}$  \\
 \hline
\cellcolor{lightgray}$I(1,4,3)$ & \cellcolor{lightgray}$\{(1,4,3)\}$  & $I(2,4,3)$ & $\{(1,4,3), (2,1,3), (2,4,1)  \}$  \\
 \hline
$I(3,4,3)$ & $\{(1,4,3),(3,4,2) \}$  & $I(4,4,3)$ & $\{(1,4,3),(4,2,3) \}$  \\
 \hline
$I(1,1,4)$ & $\{(1,2,4), (3,1,4) \}$  & $I(2,1,4)$ & $\{ (2,1,3),(2,3,4),  (3,1,4)  \}$  \\
\hline
\cellcolor{lightgray}$I(3,1,4)$ &\cellcolor{lightgray} $\{(3,1,4) \}$  & $I(4,1,4)$ & $\{(3,1,4), (4,1,2)  \}$  \\
 \hline
\cellcolor{lightgray}$I(1,2,4)$ & \cellcolor{lightgray}$\{(1,2,4) \}$  & $I(2,2,4)$ & $\{(1,2,4),(2,3,4)  \}$  \\
 \hline
$I(3,2,4)$ & $\{(1,2,4),  (3,1,4), (3,2,1) \}$  & $I(4,2,4)$ & $\{(1,2,4), (4,2,3) \}$  \\
 \hline
$I(1,3,4)$ & $\{(1,2,4), (1,3,2), (2,3,4) \}$  & \cellcolor{lightgray}$I(2,3,4)$ &\cellcolor{lightgray} $\{(2,3,4) \}$  \\
\hline
$I(3,3,4)$ & $\{(2,3,4),  (3,1,4) \}$  & $I(4,3,4)$ & $\{(2,3,4),(4,3,1) \}$  \\
 \hline
$I(1,4,4)$ & $\{(1,2,4), (1,4,3)\}$  & $I(2,4,4)$ & $\{(2,3,4),  (2,4,1) \}$  \\
 \hline
$I(3,4,4)$ & $\{(3,1,4),(3,4,2)  \}$  & $I(4,4,4)$ & $\emptyset$  \\
 \hline
\end{tabular}
}
 \caption{$I$-sets of code $C_L$, the codewords are highlighted. The empty $I$-sets belong to the diagonal vertices.}\label{C_L table}

\end{table}

\end{document}